\documentclass{amsart}
\allowdisplaybreaks
\usepackage{amssymb, amsmath, amsthm, wasysym, mathrsfs}
\usepackage[backref]{hyperref}
\usepackage[alphabetic,backrefs,lite]{amsrefs}
\usepackage{setspace}
\usepackage[all]{xy}
\usepackage{tikz}
\usepackage{fullpage}
\usepackage{mathtools}
\usepackage{enumitem}
\DeclareMathOperator{\Bl}{Bl}

\DeclareMathOperator{\Cl}{Cl}
\DeclareMathOperator{\Cone}{Cone}

\DeclareMathOperator{\DIv}{div}
\DeclareMathOperator{\Eff}{Eff}

\DeclareMathOperator{\ord}{ord}

\DeclareMathOperator{\Pic}{Pic}

\DeclareMathOperator{\rank}{rank}

\DeclareMathOperator{\Spec}{Spec}

\theoremstyle{definition}
\newtheorem{dfn}{Definition}[section]

\newtheorem{lem}[dfn]{Lemma}
\newtheorem{prop}[dfn]{Proposition}
\newtheorem{co}[dfn]{Corollary}
\newtheorem{rem}[dfn]{Remark}

\newtheorem{nota}[dfn]{Notation}
\newtheorem{ass}[dfn]{Assumption}
\theoremstyle{plain}
\newtheorem{theorem}[dfn]{Theorem}
\title{The distribution of semi-integral points on a class of singular cubic hypersurfaces}
\date{}
\author{Haruki Ito}
\address{Graduate School of Mathematics, Nagoya University, Furocho, Chikusaku, Nagoya 464-8602, Japan}
\email{haruki.ito.c8@math.nagoya-u.ac.jp}

\begin{document}

\begin{abstract}
Let $k$ be a positive integer and let $X_k$ be the cubic hypersurface defined by the equation $x^3-(y_1^2+\cdots+y_{4k}^2)z=0$.
In this paper, we give an asymptotic formula for the counting function of semi-integral points on $X_k$.
We also prove that this asymptotic formula agrees with Manin's conjecture for $\mathcal{M}$-points \cite[Conjecture~1.4]{Moe26a} with respect to the $a$-invariant and the $b$-invariant.
\end{abstract}

\maketitle
\tableofcontents
\section{Introduction}
\subsection{Main Results}

Let $k$ be a positive integer and let $X_k$ be the closed subscheme of $\mathbb{P}_{\mathbb{Q}}^{4k+1}$ defined by the equation 
\[
x^3-(y_1^2+\cdots+y_{4k}^2)z=0.
\]
In this paper, we study the distribution of semi-integral points on $X_k$.
More precisely, fixing a finite set $S$ of prime numbers, we consider rational points $P=(x:y_1:\dots:y_{4k}:z)\in X_k(\mathbb{Q})$ satisfying $v_p(z)-v_p(x)\neq1$ for every prime $p\notin S$, where $v_p$ denotes the $p$-adic valuation.
To study the distribution of semi-integral points, we introduce the height function.
For a rational point $P=(x:y_1:\dots:y_{4k}:z)\in X_k(\mathbb{Q})$, we define the height $H(P)$ by
\[
H(P)=\max\{|x|,\sqrt{y_1^2+\cdots+y_{4k}^2},|z|\},
\]
where $x,y_1,\dots,y_{4k},z$ are coprime integers.
We consider the counting function of semi-integral points on $X_k$
\[
N(B)=\#\left\{P=(x:y_1:\dots:y_{4k}:z)\in X_k(\mathbb{Q})\middle|
\begin{array}{l}
x\neq0,\\
H(P)\leq B,\\
\text{for all } p\notin S,\,v_p(z)-v_p(x)\neq1
\end{array}\right\}.
\]
The main theorem of this paper gives an asymptotic formula for $N(B)$.

\begin{theorem}\label{main}
The asymptotic formula
\[
N(B)=\frac{2k\mathcal{G}_S(1,2k-1)}{(3k-1)(4^k-1)|B_{2k}|\zeta(4k-1)}B^{4k-1}\log B+O(B^{4k-1})\quad(B\to\infty)
\]
holds, where $B_{2k}$ is the $2k$-th Bernoulli number, the constant $\mathcal{G}_S(1,2k-1)=\prod_p\mathcal{G}_p(1,2k-1)$ is positive, and the local factors $\mathcal{G}_p(1,2k-1)$ are given in Proposition~\ref{Gp}.
\end{theorem}

As shown in \S~3, the semi-integral condition for rational points on $X_k$ corresponds to the $\mathcal{M}$-condition (see \cite[Definition~4.12]{Moe26a} for the definition) of rational points on the smooth model $\widetilde{X}_k$ of $X_k$.
Hence, Theorem~\ref{main} can be interpreted as a result on the distribution of $\mathcal{M}$-points on $\widetilde{X}_k$.
Moreover, we show that the asymptotic formula in Theorem~\ref{main} agrees with Manin's conjecture for $\mathcal{M}$-points formulated in \cite[Conjecture~1.4]{Moe26a}.

\begin{prop}\label{MManin}
The asymptotic formula in Theorem~\ref{main} agrees with \cite[Conjecture~1.4]{Moe26a} with respect to the $a$-invariant and the $b$-invariant introduced in \cite[Definition~4.37]{Moe26a}.
\end{prop}

\subsection{Previous Work}

Manin's conjecture predicts an asymptotic formula for the number of rational points of bounded height on smooth Fano varieties defined over number fields.
It was proposed by Y.~Manin and his collaborators in the late 1980s and is formulated in \cite{FMT89, BM90, Pey95, BT98, Pey03, LST22}.

For an early work on the distribution of rational points on singular hypersurfaces, we refer to \cite{dlB98}. 
In \cite{LWZ19}, using the method of \cite{dlB98} together with Jacobi's four-square theorem, the distribution of rational points on the variety considered in this paper is studied. 
Furthermore, in \cite{LWZ20, Tak22, Zha22}, the distribution of rational points on similar varieties is studied by refining the computations in \cite{LWZ19}.
Studies of more general singular hypersurfaces are given in \cite{Wen23, JWZ25}.
In this paper, based on the computations in \cite{LWZ19}, we determine the distribution of semi-integral points on the variety under consideration.

Recently, semi-integral points have been extensively studied.
They are defined by conditions on intersection multiplicities with divisors on integral models.
Campana points are among the most extensively studied classes of semi-integral points.
They were introduced by F. Campana~\cite[Definition~4.1.1]{Cam05} and D. Abramovich~\cite[Definition~2.4.17]{Abr09}, and Manin's conjecture for Campana points is formulated in \cite[Conjecture~1.1]{PSTVA21} and \cite[Conjecture~8.3]{CLTBT26}.
These studies are based on the notion of Campana orbifolds developed in \cite{Cam04, Cam05, Cam11, Cam15}.
The distribution of Campana points has been studied in \cite{BVV12, VV12, BY21, PSTVA21, Shu22, Str22, Xia22, BBK24, PS24, CLTBT26}.
Related counting problems for geometric Darmon points have been studied in \cite{SS24, Ara25}.
Geometric Darmon points were introduced in \cite[Definition~2.10]{MNS24} based on ideas arising from H.~Darmon's work \cite{Dar07} on the generalized Fermat equation.
Weak Campana points are defined by more delicate conditions, and their distribution is studied in \cite{Abr09, AVA18, Str22}.

A further generalization of these notions is given by $\mathcal{M}$-points, introduced in \cite[Definition~3.12]{Moe24} and \cite[Definition~3.12]{Moe26a}. 
As explained in \cite[\S~3.3]{Moe26a}, by choosing the parameter $\mathcal{M}$ appropriately, $\mathcal{M}$-points recover several previously studied classes of semi-integral points, including Campana points, geometric Darmon points, and weak Campana points.
A conjecture on the distribution of $\mathcal{M}$-points on rationally connected smooth varieties is formulated in \cite[Conjecture~1.4]{Moe26a}.
In \cite[Theorem~3.3]{Moe26b}, the distribution of $\mathcal{M}$-points on smooth Fano toric varieties is studied, and the result is shown to agree with \cite[Conjecture~1.4]{Moe26a}.

Using the resolution $\pi\colon \widetilde{X}_k\to X_k$ constructed in \cite[\S~2]{LWZ19}, Theorem~\ref{main} can be interpreted as a result on the distribution of $\mathcal{M}$-points on $\widetilde{X}_k$.
Under this interpretation, we show that the asymptotic formula in Theorem~\ref{main} agrees with \cite[Conjecture~1.4]{Moe26a} with respect to the $a$-invariant and the $b$-invariant.
The compatibility of the leading constant with the conjectural leading constant is not established.

\subsection{Structure of the paper}
In \S~2, we introduce the notation used throughout this paper.
In \S~3, we show that the semi-integral points on $X_k$ correspond to $\mathcal{M}$-points on $\widetilde{X}_k$ (Corollary~\ref{co}).
We then determine the pseudo-effective cone of $\widetilde{X}_k$ (Proposition~\ref{eff}) and show that Theorem~\ref{main} is compatible with Manin's conjecture for $\mathcal{M}$-points (Corollary~\ref{a4k1b2}).
In \S~4, we prove Theorem~\ref{main} by refining the method of \cite{LWZ19}.

\subsection{Method}

The proof of Theorem~\ref{main} is based on the method of \cite{LWZ19}.
First, by using the inclusion-exclusion principle, we reduce the problem to counting suitable tuples $(x,y_1,\cdots,y_{4k},z)$ of integers.
Next, by fixing the values of $x$ and $y_1^2+\cdots+y_{4k}^2$, we reduce the problem to the estimation of suitable divisor sums.
We then derive the main term by analyzing the corresponding complex integrals using the method of \cite{dlB98} together with Perron's formula \cite[Theorem II.2.3]{Ten95}.

In this paper, by introducing new auxiliary functions and applying Lemma~\ref{f} to eliminate $\phi(B)$ from the error term, we obtain a smaller error term than that obtained by the previous method.

\subsection{Acknowledgements}
The author would like to thank his advisor Sho Tanimoto for continuous support and encouragement.
The author thanks Wataru Takeda for helpful comments on \cite{Tak22}.
The author thanks Boaz Moerman for helpful comments on $\mathcal{M}$-points and the structure of the paper.
The author is grateful to Shu Nimura for helpful discussions on the resolution of $X_k$ constructed in \cite{LWZ19} and its Picard group.
The author thanks Shuhei Katsuta for his careful reading of the manuscript and helpful comments.
This work was financially supported by JST SPRING, Grant Number JPMJSP2125, through the THERS Make New Standards Program for the Next Generation Researchers.

\subsection{AI Disclosure}
The author used OpenAI's ChatGPT to assist with the preparation of this manuscript.
Specifically, ChatGPT suggested a method for explicitly computing the valuations $v_i(f)$ in the study of the pseudo-effective cone $\overline{\Eff}^1(\widetilde{X}_k)$ and proposed a correction to the definition of $\delta(B)$.
The idea of introducing the auxiliary function $\phi(B)$ was developed independently by the author.
ChatGPT was also used to identify computational, grammatical, and typographical errors and to suggest corrections.
The manuscript itself was written by the author, who takes full responsibility for its mathematical correctness and final content.

\section{Notation}

Throughout this paper, we use the following notation.

\begin{nota}
Let $k$ be a positive integer and let $X_k$ be the subscheme of $\mathbb{P}_{\mathbb{Q}}^{4k+1}$ defined by the equation
\[
x^3-(y_1^2+\cdots+y_{4k}^2)z=0.
\]
For a ring $R$ and a positive integer $n$, we denote the affine space $\Spec R[x_1,\cdots,x_n]$ by $\mathbb{A}_R^n$.
\end{nota}

\begin{nota}\label{2naturalnumber}
\leavevmode
\begin{enumerate}
\item We denote the set of all non-negative integers by $\mathbb{N}$.
\item For a prime number $p$ and a non-zero integer $n$, we denote the exponent of the highest power of $p$ dividing $n$ by $v_p(n)$.
We extend $v_p$ to be a group homomorphism $v_p\colon\mathbb{Q}^\times\to\mathbb{Z}$.
\item We fix a finite set $S$ of prime numbers.
\item We denote the ring of $S$-integers of $\mathbb{Q}$ by $\mathcal{O}_S$.
\end{enumerate}
\end{nota}

\begin{nota}\label{2geometry}
For a projective variety $X$ defined over $\mathbb{Q}$, we use the following notation.
\begin{enumerate}
\item We denote the Picard group of $X$ by $\Pic(X)$.
\item We denote the cone generated by the numerical classes of effective divisors on $X$ by $\Eff^1(X)$, and its topological closure by $\overline{\Eff}^1(X)$.
\end{enumerate}
\end{nota}

\begin{nota}
\leavevmode
\begin{enumerate}
\item For functions $f,g\colon X\to\mathbb{C}$, we write $f(x)\ll g(x)$, or equivalently $f(x)=O(g(x))$, if there exists a constant $C>0$ such that
\[
|f(x)|\leq C|g(x)|
\]
for all $x\in X$.
\item For functions $f, g, h\colon X\to\mathbb{R}$, we write $f(x)+O(g(x))\leq h(x)$ if there exists a function $R\colon X\to\mathbb{R}$ such that
\[
R(x)=O(g(x))\qquad\text{and}\qquad f(x)+R(x)\leq h(x)
\]
for all $x\in X$.
The same convention applies to the reverse inequality.
\end{enumerate}
\end{nota}

\section{Compatibility with Manin's conjecture for $\mathcal{M}$-points}
In this section, we prove that Theorem~\ref{main} is compatible with Manin's conjecture for $\mathcal{M}$-points.
First, we interpret the semi-integral condition
\[
\text{for\,\,all\,\,}p\notin S, v_p(z)-v_p(x)\neq1
\]
on rational points on $X_k$ as a condition on rational points on a resolution $\widetilde{X}_k$ of $X_k$.
We set $q=y_1^2+\cdots+y_{4k}^2$.
To interpret the semi-integral condition in terms of intersection multiplicities, we introduce the following integral model of $X_k$.
Let $\mathcal{X}_k$, $\mathcal{P}$, and $\mathcal{Q}$ be the closed subschemes of $\mathbb{P}_{\mathcal{O}_S}^{4k+1}$ defined by the following equations:
\[\left\{
\begin{array}{ll}
\mathcal{X}_k&:x^3-(y_1^2+\cdots+y_{4k}^2)z=0,\\
\mathcal{P}&:x=y_1=\cdots=y_{4k}=0,\\
\mathcal{Q}&:x=y_1^2+\cdots+y_{4k}^2=z=0.
\end{array}
\right.
\]
Following \cite[\S~2]{LWZ19}, let $\pi\colon\Bl_{\mathcal{P}\sqcup\mathcal{Q}}\mathbb{P}_{\mathcal{O}_S}^{4k+1}\to\mathbb{P}_{\mathcal{O}_S}^{4k+1}$ be the blow-up of $\mathbb{P}_{\mathcal{O}_S}^{4k+1}$ along $\mathcal{P}\sqcup\mathcal{Q}$ and let $\widetilde{\mathcal{X}}_k$ be the strict transform of $\mathcal{X}_k$.
Let $\mathcal{D}_z$ and $\mathcal{D}_q$ be the strict transforms of the closed subschemes of $\mathcal{X}_k$ defined by $x=z=0$ and $x=q=0$, respectively.
Let $\mathcal{E}_1$ be the exceptional divisor of $\pi:\widetilde{\mathcal{X}}_k\to\mathcal{X}_k$ corresponding to $\mathcal{P}$.

We recall the construction of the exceptional divisors corresponding to $\mathcal{Q}$.
We consider the affine chart of $\mathcal{X}_k$ defined by $y_1\neq0$ and set $X=x/y_1$, $Y_i=y_i/y_1$, $Z=z/y_1$.
Let $\tau$ be the blow-up of $\mathbb{A}_{\mathcal{O}_S}^{4k+1}$ along $\{X=1+Y_2^2+\cdots+Y_{4k}^2=Z=0\}$.
Then the strict transform $\mathcal{X}_k'$ of $\mathcal{X}_k\cap\{y_1\neq0\}$ is given by
\[
\mathcal{X}_k'=\left\{(X,Y_2,\cdots,Y_{4k},Z,U:V:W)\in\mathbb{A}_{\mathcal{O}_S}^{4k+1}\times\mathbb{P}_{\mathcal{O}_S}^2\middle|
\begin{array}{l}
U^2X=VW,\\
XV=(1+Y_2^2+\cdots+Y_{4k}^2)U,\\
XW=UZ,\\
ZV=(1+Y_2^2+\cdots+Y_{4k}^2)W
\end{array}\right\}
\]
and the exceptional locus $\mathcal{E}$ of $\tau\colon\mathcal{X}_k'\to\mathcal{X}_k\cap\{y_1\neq0\}$ is given by
\begin{align*}
\mathcal{E}&=\{(X,Y_2,\cdots,Y_{4k},Z,U:V:W)\mid X=1+Y_2^2+\cdots+Y_{4k}^2=Z=U^2X-VW=0\}
\\
&=\{(0,Y_2,\cdots,Y_{4k},0,U:0:W)\mid 1+Y_2^2+\cdots+Y_{4k}^2=0\}
\\
&\quad\cup\{(0,Y_2,\cdots,Y_{4k},0,U:V:0)\mid 1+Y_2^2+\cdots+Y_{4k}^2=0\}.
\end{align*}
We denote by $\mathcal{E}_2$ and $\mathcal{E}_3$ the exceptional divisors of $\pi\colon\widetilde{\mathcal{X}}_k\to\mathcal{X}_k$ corresponding to the former component and the latter component, respectively.

\begin{rem}
Let $\widetilde{X}_k$ be the generic fiber of $\widetilde{\mathcal{X}}_k$.
Then the induced morphism $\pi\colon\widetilde{X}_k\to X_k$ is the resolution of $X_k$ constructed in \cite[\S~2]{LWZ19}.
\end{rem}

\begin{nota}
For each $P=(x:y_1:\dots:y_{4k}:z)\in X_k(\mathbb{Q})$ with $x\neq0$, $\widetilde{P}\in\widetilde{X}_k(\mathbb{Q})$ denotes the unique rational point such that $\pi(\widetilde{P})=P$.
\end{nota}

We recall the definition of intersection multiplicity.
This notion is used in the definition of $\mathcal{M}$-points.

\begin{dfn}
Let $Q\in\widetilde{X}_k(\mathbb Q)$, let $p\notin S$ be a prime number, and let $\mathcal{D}$ be a closed subscheme of $\widetilde{\mathcal{X}}_k$.
Since $\widetilde{\mathcal{X}}_k$ is proper over $\mathcal{O}_S$, by the valuative criterion for properness, there exists a unique $\mathbb{Z}_p$-point $\mathcal{Q}_p\in\widetilde{\mathcal{X}}_k(\mathbb{Z}_p)$ such that the following diagram commutes:
\[
\xymatrix{
\Spec\mathbb{Q}_p\ar[r]\ar[d]&\Spec\mathbb{Q}\ar[r]^Q&\widetilde{X}_k\ar[r]&\widetilde{\mathcal{X}}_k\ar[d]
\\
\Spec\mathbb{Z}_p\ar[rrr]\ar[rrru]^{\mathcal{Q}_p}&&&\Spec\mathcal{O}_S.
}
\]
The pull-back of $\mathcal{D}\subseteq\widetilde{\mathcal{X}}_k$ via $\mathcal Q_p$ defines a closed subscheme of $\Spec\mathbb Z_p$ of the form 
\[
\mathcal{D}\times_{\widetilde{\mathcal{X}}_k}\Spec\mathbb{Z}_p\cong\Spec(\mathbb Z_p/p^N\mathbb Z_p),
\]
where $N\in\mathbb{N}\cup\{\infty\}$ is uniquely determined, and we set $p^\infty\mathbb Z_p=0$.
We define
\[
n_p(\mathcal D,Q):=N
\]
and call it the intersection multiplicity of $Q$ with $\mathcal{D}$ at $p$.
\end{dfn}

\begin{prop}
Let $p\notin S$ and let $P=(x:y_1:\dots:y_{4k}:z)\in X_k(\mathbb{Q})$, where $x,y_1,\dots,y_{4k},z$ are coprime integers and $x\neq0$.
Let $A=\min\{v_p(x),v_p(q),v_p(z)\}$.
Note that $q$ and $A$ are independent of the choice of coprime integral coordinates representing $P$.
Then the following hold:
\begin{enumerate}
\item $\displaystyle n_p(\mathcal{D}_z,\widetilde{P})=\left\{
\begin{array}{ll}
v_p(x)-v_p(q)&\text{if }0<v_p(z)\text{ and }A=v_p(q),\\
0&\text{otherwise}.
\end{array}\right.$
\item $\displaystyle n_p(\mathcal{E}_3,\widetilde{P})=\left\{
\begin{array}{ll}
v_p(z)-v_p(x)&\text{if }0<v_p(z)\text{ and }A=v_p(x),\\
v_p(q)&\text{if }0<v_p(z)\text{ and }A=v_p(q),\\
0&\text{otherwise}.
\end{array}\right.$
\end{enumerate}
\end{prop}

\begin{proof}
Note that $n_p(\mathcal{D}_z,\widetilde{P})\geq0$ and $n_p(\mathcal{E}_3,\widetilde{P})\geq0$.
If $v_p(x)=0$ or $v_p(z)=0$, then by the defining equations of $\mathcal{D}_z$ and $\mathcal{E}_3$, we have
\[
n_p(\mathcal{D}_z,\widetilde{P})=n_p(\mathcal{E}_3,\widetilde{P})=0.
\]
Therefore, (1) and (2) hold in this case.

We consider the case where $v_p(x)>0$ and $v_p(z)>0$.
In this case, there exists $1\leq i\leq4k$ such that $v_p(y_i)=0$.
By symmetry, we may assume that $v_p(y_1)=0$.
Write
\[
\widetilde{P}=((x:y_1:y_2:\dots:y_{4k}:z),(u:v:w))\in(\mathbb{P}_{\mathbb{Q}}^{4k+1}\times\mathbb{P}_{\mathbb{Q}}^2)(\mathbb{Q})\subset(\mathbb{P}_{\mathbb{Q}_p}^{4k+1}\times\mathbb{P}_{\mathbb{Q}_p}^2)(\mathbb{Q}_p).
\]
Since $v_p(y_1)=0$, we may assume that $y_1=1$.
By the defining equation of $X_k$, we have
\[
3v_p(x)=v_p(x^3)=v_p((y_1^2+\cdots+y_{4k}^2)z)=v_p(q)+v_p(z).
\]
By the defining equations of $\widetilde{\mathcal{X}}_k$, we have
\[
v_p(u)=v_p(x)-A,\quad v_p(v)=v_p(q)-A,\quad v_p(w)=v_p(z)-A.
\]
\begin{enumerate}
\item Since $\mathcal{D}_z$ is defined by $x=z=u=w=0$, we have
\begin{align*}
n_p(\mathcal{D}_z,\widetilde{P})&=\min\{v_p(x),v_p(z),v_p(u),v_p(w)\}
\\
&=\min\{v_p(x),v_p(z),v_p(x)-A,v_p(z)-A\}
\\
&=\min\{v_p(x),v_p(z)\}-A
\\
&=\left\{\begin{array}{ll}
\min\{v_p(x),v_p(z)\}-v_p(q)&\text{if }A=v_p(q), \\
0&\text{otherwise}.
\end{array}\right.
\end{align*}
If $A=v_p(q)$, then $v_p(z)=3v_p(x)-v_p(q)\geq2v_p(x)>v_p(x)$, hence $n_p(\mathcal{D}_z,\widetilde{P})=v_p(x)-v_p(q)$.
\item Since $\mathcal{E}_3$ is defined by $x=1+y_2^2+\cdots+y_{4k}^2=z=w=0$, we have
\[
n_p(\mathcal{E}_3,\widetilde{P})=\min\{v_p(x),v_p(q),v_p(z)-A\}.
\]
\begin{enumerate}
\item The case $v_p(x)=A$.

Since
\begin{align*}
v_p(x)-(v_p(z)-A)&=2v_p(x)-v_p(z)=v_p(q)-v_p(x)\geq0,
\\
v_p(q)-(v_p(z)-A)&=4v_p(x)-2v_p(z)\geq0,
\end{align*}
we have $n_p(\mathcal{E}_3,\widetilde{P})=v_p(z)-A=v_p(z)-v_p(x)$.
\item The case $v_p(q)=A$.

Since
\[
(v_p(z)-A)-v_p(q)=3v_p(z)-6v_p(x)=3(v_p(x)-v_p(q))\geq0,
\]
we have $n_p(\mathcal{E}_3,\widetilde{P})=v_p(q)$.
\item The case $v_p(z)=A$.

In this case, we have $n_p(\mathcal{E}_3,\widetilde{P})=0$.\qedhere
\end{enumerate}
\end{enumerate}
\end{proof}

Let $D_z$, $D_q$, $E_1$, $E_2$, $E_3$ be the generic fibers of $\mathcal{D}_z$, $\mathcal{D}_q$, $\mathcal{E}_1$, $\mathcal{E}_2$, $\mathcal{E}_3$, respectively.
We set
\[
\mathfrak{M}=\mathbb{N}^2\setminus\{(0,1)\},\quad M=((D_z,E_3),\mathfrak{M}),\qquad\mathcal{M}=((\mathcal{D}_z,\mathcal{E}_3),\mathfrak{M}).
\]

We recall the notion of $\mathcal{M}$-points introduced in \cite[Definition~3.12]{Moe26a}.

\begin{dfn}
Let $Q\in\widetilde{X}_k(\mathbb{Q})$ be a rational point.
We say that $Q$ is an $\mathcal{M}$-point on $\widetilde{\mathcal{X}}_k$ over $\mathcal{O}_S$ if for all primes $p\notin S$,
\[
(n_p(\mathcal{D}_z,Q),n_p(\mathcal{E}_3,Q))\in\mathfrak{M}
\]
holds.
We denote the set of $\mathcal{M}$-points on $\widetilde{\mathcal{X}}_k$ over $\mathcal{O}_S$ by $(\widetilde{\mathcal{X}}_k,\mathcal{M})(\mathcal{O}_S)$.
\end{dfn}

By a direct calculation, we obtain the following corollary.

\begin{co}\label{co}
Let $p\notin S$ and let $P=(x:y_1:\dots:y_{4k}:z)\in X_k(\mathbb{Q})$ with $x\neq0$.
Then we have 
\[
2n_p(\mathcal{D}_z,\widetilde{P})+n_p(\mathcal{E}_3,\widetilde{P})=\max\{v_p(z)-v_p(x),0\}.
\]
In particular, the following conditions are equivalent.
\begin{enumerate}
\item The point $P$ satisfies the semi-integral condition 
\[
\text{for all } p\notin S,\quad v_p(z)-v_p(x)\neq1.
\]
\item The corresponding point $\widetilde{P}$ is an $\mathcal{M}$-point over $\mathcal{O}_S$.
\end{enumerate}
\end{co}

The height function in this paper
\[
H(P)=\max\{|x|,\sqrt{y_1^2+\cdots+y_{4k}^2},|z|\}
\]
is associated with the line bundle $\mathcal{O}_{X_k}(1)$.
Recall that $N(B)$ denotes the number of rational points $P\in X_k(\mathbb{Q})$ with $x(P)\neq0$ satisfying the semi-integral condition and $H(P)\leq B$.
Hence, we have
\[
N(B)=\#\{P\in(\widetilde{\mathcal{X}}_k,\mathcal{M})(\mathcal{O}_S)\mid x(\pi(P))\neq0,H_L(P)\leq B\},
\]
where $L:=\pi^{*}\mathcal{O}_{X_k}(1)$.
If \cite[Conjecture~1.4]{Moe26a} holds, there exists $c>0$ such that
\[
N(B)\sim cB^{a((\widetilde{X}_k,M),L)}(\log B)^{b(\mathbb{Q},(\widetilde{X}_k,M),L)-1}\quad(B\to\infty),
\]
where $a((\widetilde{X}_k,M),L)$ and $b(\mathbb{Q},(\widetilde{X}_k,M),L)$ are the $a$-invariant and the $b$-invariant, respectively, as defined in \cite[Definition~4.37]{Moe26a}.

By \cite[\S~4, \S~6]{Moe26a}, the computation of the $a$-invariant and the $b$-invariant reduces to the study of the pseudo-effective cone of $\widetilde{X}_k$.

\begin{prop}\label{ab}
\leavevmode
\begin{enumerate}
\item We have
\[
a((\widetilde{X}_k,M),L)=\inf\left\{t\in\mathbb{R}\middle|(t-4k+1)L+(4k-2)E_1+\frac{1}{2}E_3\in\overline{\Eff}^1(\widetilde{X}_k)\right\}.
\]
\item The $b$-invariant $b(\mathbb{Q},(\widetilde{X}_k,M),L)$ is the codimension of the minimal supported face of $\overline{\Eff}^1(\widetilde{X}_k)$ containing $(a((\widetilde{X}_k,M),L)-4k+1)L+(4k-2)E_1+(1/2)E_3$. 
\end{enumerate}
\end{prop}

\begin{proof}
As in \cite[\S~2]{LWZ19}, we have
\[
K_{\widetilde{X}_k}=(1-4k)L+(4k-2)E_1.
\]
We note that the pair $(\widetilde{X}_k,M)$ is quasi-Campana (see \cite[Definition~6.1]{Moe26a} for the definition) for the $\mathbb{Q}$-divisor $(1/2)E_3$.
The generators (see \cite[Definition~4.1]{Moe26a} for the definition) of $M$ are
\[
\Gamma_M=\{(0,2),(0,3),(1,0),(1,1)\}.
\]
The convex hull corresponding to $(1,1)\in\Gamma_M$ (see \cite[Lemma~4.3]{Moe26a} for the definition) is
\[
P_{M,(1,1)}=\{(x,y)\in\mathbb{R}^2\mid\min\{x,y,2x+y-2\}\geq0\}.
\]
The relative positions of $P_{M,(1,1)}$ and the points in $\Gamma_M$ are illustrated in the figure below.
\begin{center}
\begin{tikzpicture}[scale=0.8]
\tikzset{
  pt/.style={circle, fill=black, inner sep=1.3pt},
  axis/.style={->, semithick}
}

\fill[gray!40]
  (0,2) -- (0,4) -- (4,4) -- (4,0) -- (1,0) -- cycle;

\draw[axis] (0,0) -- (0,4.2) node[above] {$y$};
\draw[axis] (0,0) -- (4.2,0) node[right] {$x$};

\draw[thick] (0,2) -- (1,0);
\draw[thick] (0,2) -- (0,4);
\draw[thick] (1,0) -- (4,0);

\fill (1,0) circle (1.3pt) node[below] {$(1,0)$};
\fill (0,2) circle (1.3pt) node[left] {$(0,2)$};
\fill (0,3) circle (1.3pt) node[left] {$(0,3)$};
\fill (1,1) circle (1.3pt) node[right] {$(1,1)$};

\node at (2.4,2.5) {$P_{M,(1,1)}$};

\end{tikzpicture}
\end{center}
Let $\partial P_{M,(1,1)}$ be the boundary of $P_{M,(1,1)}$.
Note that
\[
\partial P_{M,(1,1)}\cap\Gamma_M=\{(1,0),(0,2),(0,3)\}.
\]
By \cite[Lemma~4.39]{Moe26a}, we have
\[
a((\widetilde{X}_k,M),L)=a((\widetilde{X}_k,M'),L),\quad b(\mathbb{Q},(\widetilde{X}_k,M),L)=b(\mathbb{Q},(\widetilde{X}_k,M'),L),
\]
where
\[
\mathfrak{M}'=\{(0,0),(1,0),(0,2),(0,3)\},\quad M'=((D_z,E_3),\mathfrak{M}').
\]
Let
\[
\mathfrak{M}_c=\{(m_1,m_2)\in\mathbb{N}^2\mid m_2\neq1\},\quad M_c=((D_z,E_3),\mathfrak{M}_c).
\]
Then the pair $(\widetilde{X}_k,M_c)$ corresponds to geometric Campana points.
By \cite[Lemma~4.39]{Moe26a}, we have
\begin{align*}
a((\widetilde{X}_k,M_c),L)&=a((\widetilde{X}_k,M'),L)=a((\widetilde{X}_k,M),L),
\\
b(\mathbb{Q},(\widetilde{X}_k,M_c),L)&=b(\mathbb{Q},(\widetilde{X}_k,M'),L)=b(\mathbb{Q},(\widetilde{X}_k,M),L).
\end{align*}
Finally, the assertion follows from \cite[Proposition~6.6, Proposition~6.10]{Moe26a}.
\end{proof}

In the following, we study the pseudo-effective cone of $\widetilde{X}_k$.
By \cite[\S~2]{LWZ19}, we have $\rank\Pic(X_k)=1$ and it follows that $\rank\Pic(\widetilde{X}_k)=4$.
In particular, $\Pic(\widetilde{X}_k)_{\mathbb{Q}}$ is generated by $L$, $E_1$, $E_2$, and $E_3$.
First, we express $D_z$ and $D_q$ in terms of $L$, $E_1$, $E_2$, and $E_3$.

\begin{prop}\label{dqdz}
We have
\[
D_q=\frac{2}{3}L-E_1-\frac{2}{3}E_2-\frac{1}{3}E_3,\quad D_z=\frac{1}{3}L-\frac{1}{3}E_2-\frac{2}{3}E_3. 
\]
\end{prop}

\begin{proof}
Since $\pi^*q\in H^0(\widetilde{X}_k,2L)$, its divisor is linearly equivalent to $2L$.
Moreover,
\[
2L=3D_q+\sum_{i=1}^3(\ord_{E_i}(\pi^{*}q))E_i.
\]
It follows that
\[
D_q=\frac{2}{3}L-\sum_{i=1}^3\frac{\ord_{E_i}(\pi^{*}q)}{3}E_i.
\]
Similarly, we have
\[
D_z=\frac{1}{3}L-\sum_{i=1}^3\frac{\ord_{E_i}(\pi^{*}z)}{3}E_i.
\]
We compute the exceptional valuations $\ord_{E_i}(\pi^{*}q)$ and $\ord_{E_i}(\pi^{*}z)$ for each $i=1,2,3$.
\begin{enumerate}
\item The case $i=1$.

Let $\pi_P:\Bl_P\mathbb{P}_{\mathbb{Q}}^{4k+1}\to\mathbb{P}_{\mathbb{Q}}^{4k+1}$ be the blow-up of $\mathbb{P}_{\mathbb{Q}}^{4k+1}$ along $P=(0:\dots:0:1)$ and let $X_k^0$ be the strict transform of $X_k$.
Then
\[
X_k^0=\left\{(x:y_1:\dots:y_{4k}:z,b_0:b_1:\dots:b_{4k})\in\mathbb{P}_{\mathbb{Q}}^{4k+1}\times\mathbb{P}_{\mathbb{Q}}^{4k}\middle|
\begin{array}{l}
b_0^2x=(b_1^2+\cdots+b_{4k}^2)z,\\
b_iy_j=b_jy_i\quad(\text{for\,\,all\,\,}i,j\in\{1,\dots,4k\}),\\
b_ix=b_0y_i\quad(\text{for\,\,all\,\,}i\in\{1,\dots,4k\})
\end{array}\right\}
\]
and the exceptional divisor $E_1$ of $\pi_P\colon X_k^0\to X_k$ is given by
\[
E_1=\{(0:\dots:0:1,b_0:b_1:\dots:b_{4k})\in\mathbb{P}_{\mathbb{Q}}^{4k+1}\times\mathbb{P}_{\mathbb{Q}}^{4k}\mid b_1^2+\cdots+b_{4k}^2=0\}.
\]
Since $\pi^{*}z$ does not vanish in a neighborhood of the generic point of $E_1$, we have $\ord_{E_1}(\pi^{*}z)=0$.
Since $\ord_{E_1}(b_0)=0$, $\ord_{E_1}(b_1^2+\cdots+b_{4k}^2)=1$, and $b_0^2x=(b_1^2+\cdots+b_{4k}^2)z$, we have $\ord_{E_1}(\pi^{*}x)=1$.
Since $x^3=qz$, we have $\ord_{E_1}(\pi^{*}q)=3$.
\item The cases $i=2$ and $i=3$.

The generic fiber $X_k'$ of $\mathcal{X}_k'$ is given by
\[
X_k'=\left\{(X,Y_2,\dots,Y_{4k},Z,U:V:W)\in\mathbb{A}_{\mathbb{Q}}^{4k+1}\times\mathbb{P}_{\mathbb{Q}}^2\middle|
\begin{array}{l}
U^2X=VW,\\
XV=(1+Y_2^2+\cdots+Y_{4k}^2)U,\\
XW=UZ,\\
ZV=(1+Y_2^2+\cdots+Y_{4k}^2)W
\end{array}
\right\}.
\]
Since the generic point of $E_2$ is contained in the open subscheme of $X_k'$ defined by $UW\neq0$, we may work on this affine chart.
Since $U^2X=VW$ and $E_2$ is defined by $V=0$, we have $\ord_{E_2}(\pi^{*}x)=1$.
Since $XW=UZ$, we have $\ord_{E_2}(\pi^{*}z)=1$ and it follows that $\ord_{E_2}(\pi^{*}q)=2$.
Similarly, we obtain $\ord_{E_3}(\pi^{*}q)=1$ and $\ord_{E_3}(\pi^{*}z)=2$.\qedhere
\end{enumerate}
\end{proof}

Finally, we determine the pseudo-effective cone of $\widetilde{X}_k$.

\begin{prop}\label{eff}
The effective cone $\Eff^1(\widetilde{X}_k)$ is generated by $D_q$, $D_z$, $E_1$, $E_2$, and $E_3$.
In particular, we have $\Eff^1(\widetilde{X}_k)=\overline{\Eff}^1(\widetilde{X}_k)$.
Moreover, $D_q$, $D_z$, $E_1$, $E_2$, and $E_3$ span extremal rays of $\Eff^1(\widetilde{X}_k)$.
\end{prop}

\begin{proof}
First, we prove the first assertion.
Let $\Gamma$ be a non-exceptional divisor of $\widetilde{X}_k$, and set $\pi_{*}\Gamma=\Delta$.
Since $\rank\Cl(X_k)=1$, there exist positive integers $r$ and $m$ such that $\mathcal{O}(r\Delta)\cong\mathcal{O}_{X_k}(m)$.
Let $f\in H^0(X_k,\mathcal{O}_{X_k}(m))$ be a section such that $\DIv(f)=r\Delta$.
We set $v_i(\cdot)=\ord_{E_i}\pi^{*}(\cdot)$ and $a_i=v_i(f)$ for each $i=1,2,3$.
Then we have
\begin{align*}
r\Gamma&=mL-a_1E_1-a_2E_2-a_3E_3
\\
&=a_1D_q+(3m-2a_1)D_z+(m-a_2)E_2+(2m-a_1-a_3)E_3,
\end{align*}
where the second equality follows from Proposition~\ref{dqdz}.
Hence, it suffices to show that
\[
2a_1\leq 3m,a_2\leq m,a_1+a_3\leq 2m.
\]
We write
\[
f=\sum_{2a+b+c+d=m,0\leq b\leq2}q(\mathbf{y})^ax^bz^ch_{a,b,c,d}(\mathbf{y}),
\]
where $h_{a,b,c,d}(\mathbf{y})$ is a homogeneous polynomial of degree $d$ that is not divisible by $q(\mathbf{y})=y_1^2+\cdots+y_{4k}^2$.
Then we have
\[
v_i(f)\geq\min\{v_i(q(\mathbf{y})^ax^bz^ch_{a,b,c,d}(\mathbf{y}))\mid 2a+b+c+d=m,0\leq b\leq2\}.
\]
We show that equality holds.
\begin{enumerate}
\item The case $i=1$.

By the proof of Proposition~\ref{dqdz}, for each term $T=q(\mathbf{y})^ax^bz^ch_{a,b,c,d}(\mathbf{y})$ of $f$, we have
\[
v_1(T)=3a+b+v_1(h_{a,b,c,d}(\mathbf{y})).
\]
We compute  $v_1(h_{a,b,c,d}(\mathbf{y}))$.
We may work on the affine chart of $X_k^0$ defined by $b_0z\neq0$.
We set $b_0=z=1$.
Then we have
\begin{align*}
f&=\sum_{2a+b+c+d=m,0\leq b\leq2}q(\mathbf{y})^ax^bz^ch_{a,b,c,d}(\mathbf{y})
\\
&=\sum_{N=0}^\infty x^N\left(\sum_{2a+b+c+d=m,0\leq b\leq 2,3a+b+d=N}h_{a,b,c,d}(b_1,\dots,b_{4k})\right)
\\
&=:\sum_{N=0}^\infty f_N.
\end{align*}
Hence, it suffices to show that $v_1(f_N)=N$ for all $N$ such that $f_N\neq0$.
Suppose that $f_N\neq0$ and $v_1(f_N)\neq N$.
Then, for all $d$, we have
\[
\sum_{2a+b+c=m-d,0\leq b\leq2,3a+b=N-d}h_{a,b,c,d}(b_1,\dots,b_{4k})\in(b_1^2+\cdots+b_{4k}^2).
\]
However, for each $d$, there is at most one tuple $(a,b,c)\in\mathbb{Z}^3$ satisfying
\[
2a+b+c=m-d,\quad0\leq b\leq2,\quad3a+b=N-d.
\]
It follows that
\[
h_{a,b,c,d}(b_1,\dots,b_{4k})\in(b_1^2+\cdots+b_{4k}^2)
\]
for some $a$, $b$, $c$, and $d$.
This is a contradiction.
Therefore, we obtain
\[
v_1(f)=\min\{3a+b+d\mid 2a+b+c+d=m,0\leq b\leq 2,h_{a,b,c,d}(\mathbf{y})\neq0\}.
\]
\item The cases $i=2$ and $i=3$.

Since the computation for $i=3$ is similar, it suffices to consider the case $i=2$.
We may work on the affine chart of $X_k'$ defined by $U\neq0$.
We set $U=1$.
Then we have
\[
X=VW,\quad1+Y_2^2+\cdots+Y_{4k}^2=XV=V^2W,\quad Z=XW=VW^2.
\]
Hence,
\begin{align*}
f&=\sum_{2a+b+c+d=m,0\leq b\leq2}V^{2a+b+c}W^{a+b+2c}h_{a,b,c,d}(1,Y_2,\dots,Y_{4k})
\\
&=\sum_{N=0}^\infty V^N\left(\sum_{2a+b+c+d=m,0\leq b\leq2,2a+b+c=N}W^{a+b+2c}h_{a,b,c,d}(1,Y_2,\dots,Y_{4k})\right)
\\
&=:\sum_{N=0}^\infty g_N.
\end{align*}
It suffices to show that $v_2(g_N)=N$ for all $N$ such that $g_N\neq 0$.
Suppose that $g_N\neq0$ and $v_2(g_N)\neq N$.
Then
\[
\sum_{2a+b+c+d=m,0\leq b\leq2,2a+b+c=N}W^{a+b+2c}h_{a,b,c,d}(1,Y_2,\dots,Y_{4k})=0
\]
in
\[
\left(\frac{\mathbb{Q}[Y_2,\dots,Y_{4k},V,W]}{(1+Y_2^2+\cdots+Y_{4k}^2-V^2W)}\right)/(V)\cong\frac{\mathbb{Q}[Y_2,\dots,Y_{4k}]}{(1+Y_2^2+\cdots+Y_{4k}^2)}\otimes_{\mathbb{Q}}\mathbb{Q}[W].
\]
Since the values $a+b+2c$ are distinct among the terms of $g_N$, we have
\[
h_{a,b,c,d}(1,Y_2,\dots,Y_{4k})\in(1+Y_2^2+\cdots+Y_{4k}^2).
\]
This is a contradiction.
\end{enumerate}
We next show that $2a_1\leq 3m$, $a_2\leq m$, $a_1+a_3\leq 2m$.
By (1), there exist $a$, $b$, $c$, and $d$ such that
\[
a_1=3a+b+d,\quad m=2a+b+c+d.
\]
It follows that $3m-2a_1=b+3c+d\geq0$.
By (2), there exist $a$, $b$, $c$, and $d$ such that
\[
a_2=2a+b+c\leq 2a+b+c+d=m.
\]
By (1) and (2), for all $a$, $b$, $c$, and $d$ such that $h_{a,b,c,d}(\mathbf{y})\neq0$, we have
\[
a_1\leq 3a+b+d,\quad a_3\leq a+b+2c,\quad m=2a+b+c+d.
\]
It follows that $a_1+a_3\leq2m-d\leq 2m$.
Therefore, $\Eff^1(\widetilde{X}_k)$ is generated by $D_q$, $D_z$, $E_1$, $E_2$, and $E_3$.
Hence, $\Eff^1(\widetilde{X}_k)=\overline{\Eff}^1(\widetilde{X}_k)$.
By Proposition~\ref{dqdz} and the fact that $L$, $E_1$, $E_2$, and $E_3$ form a basis of $\Pic(\widetilde{X}_k)_{\mathbb{R}}$, one easily checks that $D_q$, $D_z$, $E_1$, $E_2$, $E_3$ span extremal rays of $\overline{\Eff}^1(\widetilde{X}_k)$.
\end{proof}

We compute the $a$-invariant and the $b$-invariant.

\begin{co}\label{a4k1b2}
We have
\[
a((\widetilde{X}_k,M),L)=4k-1,\quad b(\mathbb{Q},(\widetilde{X}_k,M),L)=2.
\]
\end{co}

\begin{proof}
The first assertion follows from Proposition~\ref{dqdz} and Proposition~\ref{eff}.
Indeed, the $\mathbb{Q}$-divisor 
\[
\{(4k-1)-(4k-1)\}L+(4k-2)E_1+\frac{1}{2}E_3=(4k-2)E_1+\frac{1}{2}E_3
\]
lies on the boundary of $\Eff^1(\widetilde{X}_k)$.
The second assertion follows from the fact that the set
\[
\Cone(E_1,E_3)=\overline{\Eff}^1(\widetilde{X}_k)\cap\{aL+bE_1+cE_2+dE_3\in\Pic(\widetilde{X}_k)_{\mathbb{R}}\mid 2a+c=0\}
\]
is the minimal supported face of $\overline{\Eff}^1(\widetilde{X}_k)$ containing $(4k-2)E_1+(1/2)E_3$.
\end{proof}

\section{Proof of the main theorem}
Let $k$ be a positive integer and let $X_k$ be the closed subscheme of $\mathbb{P}_{\mathbb{Q}}^{4k+1}$ defined by the equation
\[
x^3-(y_1^2+\cdots+y_{4k}^2)z=0.
\]
For a rational point $P=(x:y_1:\dots:y_{4k}:z)\in X_k(\mathbb{Q})$, we define the height $H(P)$ by
\[
H(P)=\max\{|x|,\sqrt{y_1^2+\cdots+y_{4k}^2},|z|\},
\]
where $x,y_1,\dots,y_{4k},z$ are coprime integers.
In this section, we fix a finite set $S$ of primes and establish an asymptotic formula for the counting function of semi-integral points on $X_k$ defined by
\[
N(B)=\#\left\{P=(x:y_1:\dots:y_{4k}:z)\in X_k(\mathbb{Q})\middle|
\begin{array}{l}
x\neq0,\\
H(P)\leq B,\\
\text{for all } p\notin S,\,v_p(z)-v_p(x)\neq1
\end{array}\right\},
\]
using a method similar to that of \cite{LWZ19}.

We set
\[
N^{*}(B)=\#\left\{(x,y_1,\dots,y_{4k},z)\in\mathbb{Z}^{4k+2}\middle|
\begin{array}{l}
x^3-(y_1^2+\cdots+y_{4k}^2)z=0,\\
x\neq0,\\
\max\{|x|,\sqrt{y_1^2+\cdots+y_{4k}^2},|z|\}\leq B,\\
\text{for all\,\,}p\notin S,v_p(z)-v_p(x)\neq1
\end{array}\right\}.
\]

By the inclusion-exclusion principle, we have
\[
N(B)=\frac{1}{2}\sum_{d=1}^\infty\mu(d)N^{*}\left(\frac{B}{d}\right)=\frac{1}{2}\sum_{d\leq B}\mu(d)N^{*}\left(\frac{B}{d}\right),
\]
where $\mu$ is the M\"obius function.
We calculate $N^{*}$ by fixing the values of
\[
n:=x \quad \text{and} \quad d:=y_1^2+\cdots+y_{4k}^2.
\]
For a positive integer $d$, we define
\[
r_{4k}(d)=\#\{(y_1,\dots,y_{4k})\in\mathbb{Z}^{4k}\mid y_1^2+\cdots+y_{4k}^2=d\}.
\]
We define a map $\mathbf{1}_S\colon\mathbb{Q}^\times\to\{0,1\}$ by
\[
\mathbf{1}_S(x)=\left\{\begin{array}{ll}
1&\text{if } v_p(x)\neq1 \text{ for all } p\notin S,
\\
0&\text{otherwise}.
\end{array}\right.
\]
Taking the two possible signs of $x$ into account, we obtain
\[
N^{*}(B)=2\sum_{n\leq B}\sum_{\substack{n^3/B\leq d\leq B^2\\ d\mid n^3}}r_{4k}(d)\mathbf{1}_S\left(\frac{n^2}{d}\right)=2\sum_{n\leq B}\left(\sum_{\substack{d\leq B^2\\ d\mid n^3}}-\sum_{\substack{d<n^3/B\\ d\mid n^3}}\right)r_{4k}(d)\mathbf{1}_S\left(\frac{n^2}{d}\right).
\]
Here we have used the identity
\[
v_p(z)-v_p(x)=v_p(n^2/d)
\]
for $P=(x,y_1,\dots,y_{4k},z)\in\mathbb{Z}^{4k+2}$ satisfying $x^3=(y_1^2+\cdots+y_{4k}^2)z$ and $x\neq0$.
We define a multiplicative function $r_{4k}^{*}\colon\mathbb{Z}_{\geq1}\to\mathbb{R}$ such that
\begin{equation}\label{1}
r_{4k}^{*}(p^l)=\left\{
\begin{array}{ll}
\displaystyle\left(1-\frac{(-1)^k}{1-2^{2k-1}}\right)2^{l(2k-1)}-(-1)^k\frac{1-2^{2k}}{1-2^{2k-1}}&\text{if }p=2\text{ and }l\geq1,
\\
\displaystyle\frac{1-p^{(l+1)(2k-1)}}{1-p^{2k-1}}&\text{otherwise}
\end{array}\right.
\end{equation}
for all primes $p$ and for all non-negative integers $l$.
As shown in \cite[(2.2)]{LWZ19}, we have
\[
r_{4k}(d)=\frac{4k}{(4^k-1)|B_{2k}|}r_{4k}^{*}(d)+O(d^k)\quad(d\to\infty),
\]
where $B_{2k}$ is the $2k$-th Bernoulli number.
We note that the error term vanishes when $k=1$ (cf.~\cite[(3.9)]{Gro85}).
By an argument similar to that in \cite[page~3]{Tak22}, we have
\[
\sum_{n\leq B}\left(\sum_{\substack{d\leq B^2\\ d\mid n^3}}-\sum_{\substack{d<n^3/B\\ d\mid n^3}}\right)O(d^k)=O(B^{2k+1}(\log B)^3)=O(B^{4k-\frac{3}{2}}(\log B)^3)\quad(B\to\infty)
\]
when $k\geq2$.
For $X$, $Y$, $B\in\mathbb{R}_{>0}$, we define
\begin{align*}
S(X,Y)&=\sum_{n\leq X}\sum_{\substack{d\leq Y\\ d\mid n^3}}r_{4k}^{*}(d)\mathbf{1}_S\left(\frac{n^2}{d}\right),
\\
T(B)&=\sum_{n\leq B}\sum_{\substack{d<n^3/B\\ d\mid n^3}}r_{4k}^{*}(d)\mathbf{1}_S\left(\frac{n^2}{d}\right).
\end{align*}
Then we have
\[
N^{*}(B)=\frac{8k}{(4^k-1)|B_{2k}|}(S(B,B^2)-T(B))+O(B^{4k-\frac{3}{2}}(\log B)^3)\quad(B\to\infty).
\]
To evaluate $T$ using $S$, we introduce an auxiliary function $\phi\colon\mathbb{R}_{>0}\to\mathbb{R}_{>0}$ satisfying the following conditions.

\begin{ass}\label{3ass}
\leavevmode
\begin{enumerate}
\item $\phi$ is strictly increasing,
\item $\lim_{x\to\infty}\phi(x)=\infty$,
\item $\phi(x)\leq\sqrt{\log(x+3)}$ for all $x\in\mathbb{R}_{>0}$.
\end{enumerate}
\end{ass}

\begin{nota}
In this section, we use the following notation.
\begin{enumerate}
\item $\varepsilon=32^{-1}$,
\item let $B>e^{32}$,
\item $\delta=1-\phi(B)(\log B)^{-1}$,
\item let $k_0$ be a positive integer such that $\delta^{k_0}<B^{-\varepsilon}\leq\delta^{k_0-1}$,
\item $\kappa=1+(\log B)^{-1}$,
\item $\lambda=2k-1+(\log B)^{-1}$.
\end{enumerate}
\end{nota}

The assumption $B>e^{32}$ is needed for Proposition~\ref{ER1}.

\begin{lem}\label{STS}
By an argument similar to that in \cite[\S~6]{LWZ19}, the following inequalities hold:
\begin{enumerate}
\item $\displaystyle\sum_{l=1}^{k_0}(S(\delta^{l-1}B,\delta^{3l}B^2)-S(\delta^lB,\delta^{3l}B^2))\leq T(B)$,
\item $\displaystyle T(B)\leq S(\delta^{k_0}B,\delta^{3k_0}B^2)+\sum_{l=1}^{k_0}(S(\delta^{l-1}B,\delta^{3(l-1)}B^2)-S(\delta^lB,\delta^{3(l-1)}B^2))$.
\end{enumerate}
\end{lem}

First, we estimate the terms corresponding to $l=k_0$.
By an argument similar to that in \cite[\S~6]{LWZ19}, for $x\geq2$ and $y\geq1$, we have
\[
S(x,y)\ll xy^{2k-1}(\log x)^{15}.
\]
Hence, we have the following lemma.

\begin{lem}\label{S}
With the above notation, we have the following estimates as $B\to\infty$:
\begin{enumerate}
\item $S(\delta^{k_0}B,\delta^{3k_0}B^2)\ll 
B^{4k-1-4\varepsilon}(\log B)^{30}$,
\item $S(\delta^{k_0-1}B,\delta^{3k_0}B^2)-S(\delta^{k_0}B,\delta^{3k_0}B^2)\ll 
B^{4k-1-4\varepsilon}(\log B)^{30}$,
\item $S(\delta^{k_0-1}B,\delta^{3(k_0-1)}B^2)-S(\delta^{k_0}B,\delta^{3(k_0-1)}B^2)\ll 
B^{4k-1-4\varepsilon}(\log B)^{30}$.
\end{enumerate}
\end{lem}

\begin{proof}
We note that $k_0=O(\phi(B)^{-1}(\log B)^2)=O((\log B)^2)$.
Hence,
\[
S(\delta^{k_0}B,\delta^{3k_0}B^2)\ll(\delta^{k_0})^{6k-2}B^{4k-1}(k_0\log\delta+\log B)^{15}\ll B^{4k-1-4\varepsilon}(\log B)^{30}.
\]
This proves (1).
The proofs of (2) and (3) are similar.
\end{proof}
By Lemma~\ref{STS} and Lemma~\ref{S}, we obtain the inequalities
\begin{align*}
\sum_{l=1}^{k_0-1}\left(S(\delta^{l-1}B,\delta^{3l}B^2)-S(\delta^lB,\delta^{3l}B^2)\right)+O(B^{4k-1-4\varepsilon}(\log B)^{30})&\leq T(B),
\\
\sum_{l=1}^{k_0-1}\left(S(\delta^{l-1}B,\delta^{3(l-1)}B^2)-S(\delta^lB,\delta^{3(l-1)}B^2)\right)+O(B^{4k-1-4\varepsilon}(\log B)^{30})&\geq T(B).
\end{align*}

Next, following \cite{dlB98}, we define
\[
M(X,Y)=\int_0^X\int_0^YS(x,y)\,\mathrm{d}y\mathrm{d}x
\]
to evaluate the remaining terms.
For convenience, we introduce the following notation.

\begin{nota}\label{nota}
Let $f\colon\mathbb{R}^2\to\mathbb{C}$ be a function.
\begin{enumerate}
\item We define a function $\mathscr{D}f\colon\mathbb{R}^4\to\mathbb{C}$ by
\[
\mathscr{D}f(X_1,X_2;Y_1,Y_2)=f(X_1,Y_1)+f(X_2,Y_2)-f(X_1,Y_2)-f(X_2,Y_1).
\]
\item We define functions $E^{+}f, E^{-}f\colon\mathbb{R}_{>0}^2\to\mathbb{C}$ by
\[
E^{\pm}f(X,Y)=\frac{\mathscr{D}f(X,X\pm X^{1-\varepsilon};Y,Y\pm Y^{1-\varepsilon})}{X^{1-\varepsilon}Y^{1-\varepsilon}}.
\]
\item For each $i=1,2,3,4$, we define a function $E_if\colon\mathbb{R}_{>e^{32}}\to\mathbb{C}$ by
\begin{align*}
E_1f(B)&=\sum_{l=1}^{k_0-1}E^{+}f(\delta^{l-1}B,\delta^{3(l-1)}B^2),
\\
E_2f(B)&=\sum_{l=1}^{k_0-1}E^{-}f(\delta^{l}B,\delta^{3(l-1)}B^2),
\\
E_3f(B)&=\sum_{l=1}^{k_0-1}E^{+}f(\delta^{l}B,\delta^{3l}B^2),
\\
E_4f(B)&=\sum_{l=1}^{k_0-1}E^{-}f(\delta^{l-1}B,\delta^{3l}B^2).
\end{align*}
\end{enumerate}
\end{nota}

As in \cite[Lemma~2]{dlB98}, we have the following lemma.

\begin{lem}\label{EMS}
For $X, Y>1$, the inequality
\[
E^{-}M(X,Y)\leq S(X,Y)\leq E^{+}M(X,Y)
\]
holds.
\end{lem}

\begin{proof}
Since
\[
S(X,Y)=\sum_{n\leq X}\sum_{\substack{d\leq Y\\ d\mid n^3}}r_{4k}^{*}(d)\mathbf{1}_S\left(\frac{n^2}{d}\right)
\]
is non-decreasing in $X$ and $Y$, we have
\begin{align*}
X^{1-\varepsilon}Y^{1-\varepsilon}S(X,Y)&=\int_X^{X+X^{1-\varepsilon}}\int_Y^{Y+Y^{1-\varepsilon}}S(X,Y)\,\mathrm{d}y\mathrm{d}x
\\
&\leq\int_X^{X+X^{1-\varepsilon}}\int_Y^{Y+Y^{1-\varepsilon}}S(x,y)\,\mathrm{d}y\mathrm{d}x
\\
&=\mathscr{D}M(X,X+X^{1-\varepsilon};Y,Y+Y^{1-\varepsilon}).
\end{align*}
It follows that $S(X,Y)\leq E^+M(X,Y)$.
We can prove $S(X,Y)\geq E^-M(X,Y)$ similarly.
\end{proof}

By Lemma~\ref{STS}, Lemma~\ref{S}, and Lemma~\ref{EMS}, we have the following corollary.

\begin{co}\label{Cor}
The inequality
\[
E_4M(B)-E_3M(B)+O(B^{4k-1-4\varepsilon}(\log B)^{30})\leq T(B)\leq E_1M(B)-E_2M(B)+O(B^{4k-1-4\varepsilon}(\log B)^{30})
\]
holds.
\end{co}

To evaluate $M(X,Y)$, we introduce the double Dirichlet series
\[
\mathcal{F}(s,w)=\sum_{n=1}^\infty\sum_{d\mid n^3}\frac{r_{4k}^{*}(d)}{n^sd^w}\mathbf{1}_S\left(\frac{n^2}{d}\right),
\]
where $r_{4k}^{*}$ is the multiplicative function defined in \eqref{1}.
As in \cite[Lemma~5.2]{LWZ19}, we obtain the following proposition by applying Perron's formula \cite[Theorem II.2.3]{Ten95} twice.

\begin{prop}\label{Perron}
We have
\[
M(X,Y)=\frac{1}{(2\pi i)^2}\int_{\kappa-iB^{8k}}^{\kappa+iB^{8k}}\int_{\lambda-iB^{8k}}^{\lambda+iB^{8k}}\frac{\mathcal{F}(s,w)X^{s+1}Y^{w+1}}{s(s+1)w(w+1)}\,\mathrm{d}w\mathrm{d}s+O(1)
\]
as $B\to\infty$, uniformly for $B^{1-\varepsilon}/2\leq X\leq 2B$ and $B^{2-3\varepsilon}/2\leq Y\leq 2B^2$, where $\kappa=1+(\log B)^{-1}$ and $\lambda=2k-1+(\log B)^{-1}$.
\end{prop}

As in \cite[\S~3]{LWZ19}, we derive an expression for $\mathcal{F}(s,w)$ in terms of zeta functions by computing certain rational functions.
Since $r_{4k}^{*}$ is multiplicative, we formally have the following Euler product:
\begin{align*}
\mathcal{F}(s,w)&=\left(\prod_{p\notin S}\sum_{a=0}^\infty p^{-as}\sum_{b\in\{0,\dots,3a\}\setminus\{2a-1\}}p^{-bw}r_{4k}^{*}(p^b)\right)\left(\prod_{p\in S}\sum_{a=0}^\infty p^{-as}\sum_{b\in\{0,\dots,3a\}}p^{-bw}r_{4k}^{*}(p^b)\right)
\\
&=:\prod_p\mathcal{F}_p(s,w).
\end{align*}

\begin{prop}
For each prime $p$, the series $\mathcal{F}_p$ defines a holomorphic function on the domain
\[
R:=\left\{(s,w)\in\mathbb{C}^2\middle|\Re s>\frac{15}{16},\quad\Re w>2k-\frac{17}{16}\right\}.
\]
\end{prop}

\begin{proof}
Let $x=p^{-s}$, let $y=p^{-w}$, and let $z=p^{2k-1}$.
\begin{enumerate}
\item The case $p\in S\setminus\{2\}$.

In this case, the series $\mathcal{F}_p$ agrees with that in the proof of \cite[Lemma~7.1]{LWZ19}.
We have
\[
\mathcal{F}_p(s,w)=\sum_{a=0}^\infty p^{-as}\sum_{b=0}^{3a}p^{-bw}r_{4k}^{*}(p^b)=\sum_{a=0}^\infty x^a\sum_{b=0}^{3a}y^b\frac{1-p^{(2k-1)(b+1)}}{1-p^{2k-1}}.
\]
Hence, we obtain
\begin{align*}
&(1-z)\mathcal{F}_p(s,w)
\\
&=\sum_{a=0}^\infty x^a\sum_{b=0}^{3a}y^b(1-z^{b+1})
\\
&=\sum_{a=0}^\infty x^a\left(\frac{1-y^{3a+1}}{1-y}-\frac{z(1-y^{3a+1}z^{3a+1})}{1-yz}\right)
\\
&=\frac{1}{1-y}\sum_{a=0}^\infty(x^a-x^ay^{3a+1})-\frac{z}{1-yz}\sum_{a=0}^\infty(x^a-x^ay^{3a+1}z^{3a+1})
\\
&=\frac{1}{1-y}\left(\frac{1}{1-x}-\frac{y}{1-xy^3}\right)-\frac{z}{1-yz}\left(\frac{1}{1-x}-\frac{yz}{1-xy^3z^3}\right)
\\
&=\frac{(1+xy+xyz+xy^2+xy^2z+xy^2z^2+xy^3z+xy^3z^2+x^2y^4z^2)(1-z)}{(1-x)(1-xy^3)(1-xy^3z^3)}.
\end{align*}
It follows that
\begin{align*}
&\mathcal{F}_p(s,w)
\\
&=\frac{1+xy+xyz+xy^2+xy^2z+xy^2z^2+xy^3z+xy^3z^2+x^2y^4z^2}{(1-x)(1-xy^3)(1-xy^3z^3)}
\\
&=(1+p^{-s-w}+p^{-s-w+2k-1}+p^{-s-2w}+p^{-s-2w+2k-1}+p^{-s-2w+4k-2}
\\
&\quad+p^{-s-3w+2k-1}+p^{-s-3w+4k-2}+p^{-2s-4w+4k-2})(1-p^{-s})^{-1}(1-p^{-s-3w})^{-1}(1-p^{-s-3w+6k-3})^{-1}.
\end{align*}
Hence, $\mathcal{F}_p(s,w)$ is holomorphic on $R$.
\item The case $p\notin S\cup\{2\}$.

We have
\begin{align*}
&\mathcal{F}_p(s,w)
\\
&=\sum_{a=0}^\infty p^{-as}\sum_{b\in\{0,\dots,3a\}\setminus\{2a-1\}}p^{-bw}r_{4k}^{*}(p^b)
\\
&=\sum_{a=0}^\infty x^a\sum_{b\in\{0,\dots,3a\}\setminus\{2a-1\}}y^b\frac{1-p^{(2k-1)(b+1)}}{1-p^{2k-1}}
\\
&=\sum_{a=0}^\infty x^a\sum_{b\in\{0,\dots,3a\}\setminus\{2a-1\}}y^b\frac{1-z^{b+1}}{1-z}.
\end{align*}
Hence, we obtain
\begin{align*}
&(1-z)\mathcal{F}_p(s,w)
\\
&=\sum_{a=0}^\infty x^a\sum_{b\in\{0,\dots,3a\}\setminus\{2a-1\}}y^b(1-z^{b+1})
\\
&=\sum_{a=0}^\infty x^a\sum_{b=0}^{3a}y^b(1-z^{b+1})-\sum_{a=1}^\infty x^ay^{2a-1}(1-z^{2a})
\\
&=\frac{(1+xy+xyz+xy^2+xy^2z+xy^2z^2+xy^3z+xy^3z^2+x^2y^4z^2)(1-z)}{(1-x)(1-xy^3)(1-xy^3z^3)}-\frac{xy}{1-xy^2}+\frac{xyz^2}{1-xy^2z^2}
\\
&=(1-z)\left(\frac{1+xy+xyz+xy^2+xy^2z+xy^2z^2+xy^3z+xy^3z^2+x^2y^4z^2}{(1-x)(1-xy^3)(1-xy^3z^3)}-\frac{xy(1+z)}{(1-xy^2)(1-xy^2z^2)}\right).
\end{align*}
It follows that
\begin{align*}
&\mathcal{F}_p(s,w)
\\
&=\frac{1+xy+xyz+xy^2+xy^2z+xy^2z^2+xy^3z+xy^3z^2+x^2y^4z^2}{(1-x)(1-xy^3)(1-xy^3z^3)}-\frac{xy(1+z)}{(1-xy^2)(1-xy^2z^2)}
\\
&=\frac{1+F(x,y,z)}{(1-x)(1-xy^3)(1-xy^2)(1-xy^3z^3)(1-xy^2z^2)}
\\
&=\frac{1+F(p^{-s},p^{-w},p^{2k-1})}{(1-p^{-s})(1-p^{-s-3w})(1-p^{-s-2w})(1-p^{-s-3(w-2k+1)})(1-p^{-s-2(w-2k+1)})},
\end{align*}
where the polynomial $F(X,Y,Z)$ belongs to
\[
\bigoplus_{\substack{(a,b,c)\in\mathbb{Z}_{\geq0}^3,\\ a+(2k-1)(b-c)\geq2,\\ a+b\leq12}}\mathbb{Z}X^aY^bZ^c.
\]
Concretely, the polynomial $F(X,Y,Z)$ is given as follows:
\begin{align*}
&F(X,Y,Z)
\\
&=X^2YZ-X^2Y^3Z^3-X^3Y^4Z^4
\\
&\quad+XY^2Z+XY^3Z^2+X^2Y-X^2Y^3Z^2-X^2Y^5Z^4-X^3Y^4Z^3
\\
&\quad+XY^3Z-X^2Y^3Z-X^2Y^5Z^3+X^3Y^5Z^3
\\
&\quad-X^2Y^3-X^2Y^5Z^2-X^3Y^4Z+X^3Y^5Z^2+X^3Y^6Z^3+X^4Y^7Z^4
\\
&\quad-X^2Y^5Z-X^3Y^4+X^4Y^7Z^3+X^4Y^8Z^4.
\end{align*}
Hence, $\mathcal{F}_p(s,w)$ is holomorphic on $R$.
\item The case $p=2\in S$.

In this case, the series $\mathcal{F}_2$ coincides with that in the proof of \cite[Lemma~7.1]{LWZ19}.
Let
\[
A=1-\frac{(-1)^k}{1-2^{2k-1}},\qquad C=-(-1)^k\frac{1-2^{2k}}{1-2^{2k-1}}.
\]
Hence, we have
\begin{align*}
&\mathcal{F}_2(s,w)
\\
&=\sum_{a=0}^\infty x^a\sum_{b=0}^{3a}y^br_{4k}^{*}(2^b)
\\
&=\sum_{a=0}^\infty x^a\left(1+\sum_{b=1}^{3a}y^b(Az^b+C)\right)
\\
&=\frac{1}{1-x}+\frac{xyz(1+yz+y^2z^2)}{(1-x)(1-xy^3z^3)}A+\frac{xy(1+y+y^2)}{(1-x)(1-xy^3)}C
\\
&=\frac{1}{1-2^{-s}}+\frac{2^{-s-(w-2k+1)}(1+2^{-(w-2k+1)}+2^{-2(w-2k+1)})}{(1-2^{-s})(1-2^{-s-3(w-2k+1)})}\left(1-\frac{(-1)^k}{1-2^{2k-1}}\right)
\\
&\quad-\frac{2^{-s-w}(1+2^{-w}+2^{-2w})(-1)^k(1-2^{2k})}{(1-2^{-s})(1-2^{-s-3w})(1-2^{2k-1})}.
\end{align*}
Therefore, $\mathcal{F}_2(s,w)$ is holomorphic on $R$.

\item The case $p=2\notin S$.

Define $A$ and $C$ as above.
Then we have
\begin{align*}
&\mathcal{F}_2(s,w)
\\
&=\sum_{a=0}^\infty x^a\sum_{b=0}^{3a}y^br_{4k}^{*}(2^b)-\sum_{a=1}^\infty x^ay^{2a-1}r_{4k}^{*}(2^{2a-1})
\\
&=\frac{1}{1-x}+\frac{xyz(1+yz+y^2z^2)}{(1-x)(1-xy^3z^3)}A+\frac{xy(1+y+y^2)}{(1-x)(1-xy^3)}C-\frac{xyz}{1-xy^2z^2}A-\frac{xy}{1-xy^2}C
\\
&=\frac{1}{1-x}+\frac{xyz(x+yz+y^2z^2-xy^2z^2-xy^4z^4-x^2y^3z^3)}{(1-x)(1-xy^2z^2)(1-xy^3z^3)}A+\frac{xy(x+y+y^2-xy^2-xy^4-x^2y^3)}{(1-x)(1-xy^2)(1-xy^3)}C
\\
&=\frac{1}{1-2^{-s}}-\frac{2^{-s-w}(2^{-s}+2^{-w}+2^{-2w}-2^{-s-2w}-2^{-s-4w}-2^{-2s-3w})(-1)^k(1-2^{2k})}{(1-2^{-s})(1-2^{-s-2w})(1-2^{-s-3w})(1-2^{2k-1})}
\\
&+\frac{2^{-s-(w-2k+1)}(2^{-s}+2^{-(w-2k+1)}+2^{-2(w-2k+1)}-2^{-s-2(w-2k+1)}-2^{-s-4(w-2k+1)}-2^{-2s-3(w-2k+1)})}{(1-2^{-s})(1-2^{-s-2(w-2k+1)})(1-2^{-s-3(w-2k+1)})}
\\
&\quad\times\left(1-\frac{(-1)^k}{1-2^{2k-1}}\right).
\end{align*}
Therefore, $\mathcal{F}_2(s,w)$ is holomorphic on $R$.
\end{enumerate}
\end{proof}

By the expression for $\mathcal{F}_p(s,w)$ when $p\notin S\cup\{2\}$, the Euler product
\[
\mathcal{F}(s,w)=\prod_p\mathcal{F}_p(s,w)
\]
converges for $\Re s>1$ and $\Re w>2k-1$.
To study this Euler product, we introduce the following function $\mathcal{G}_p(s,w)$.

\begin{dfn}\label{G}
For each prime $p$, we define a holomorphic function $\mathcal{G}_p$ on $R$ by
\[
\mathcal{G}_p(s,w)=(1-p^{-s})(1-p^{-s-2(w-2k+1)})(1-p^{-s-3(w-2k+1)})\mathcal{F}_p(s,w).
\]
\end{dfn}

\begin{prop}\label{PG}
For every compact subset $K$ of $R$, we have
\[
\mathcal{G}_p(s,w)=1+O(p^{-5/4})\quad(p\to\infty)
\]
uniformly for $(s,w)\in K$.
Hence, the Euler product
\[
\mathcal{G}_S(s,w):=\prod_p\mathcal{G}_p(s,w)
\]
converges uniformly on compact subsets of $R$ and defines a holomorphic function on $R$.
\end{prop}

\begin{proof}
Since $S\cup\{2\}$ is a finite set, we may assume that $p\notin S\cup\{2\}$.
Then we have
\[
\mathcal{G}_p(s,w)=\frac{1+F(p^{-s},p^{-w},p^{2k-1})}{(1-p^{-s-2w})(1-p^{-s-3w})}
\]
and one can check that $(1-p^{-s-2w})^{-1}(1-p^{-s-3w})^{-1}=1+O(p^{-45/16})$.
Furthermore, we have
\begin{align*}
F(p^{-s},p^{-w},p^{2k-1})&=\sum_{\substack{a,b,c\in\mathbb{Z}_{\geq0},\\a+(2k-1)(b-c)\geq2,\\ a+b\leq12}}O(p^{-a\Re s-b\Re w+c(2k-1)})
\\
&=\sum_{\substack{a,b,c\in\mathbb{Z}_{\geq0},\\a+(2k-1)(b-c)\geq2,\\ a+b\leq12}}O(p^{-a(\Re s-1)-b(\Re w-2k+1)-a-(2k-1)(b-c)})
\\
&=\sum_{\substack{a,b,c\in\mathbb{Z}_{\geq0},\\a+(2k-1)(b-c)\geq2,\\ a+b\leq12}}O(p^{\frac{a+b}{16}-2})
\\
&=O(p^{-5/4})
\end{align*}
uniformly for $(s,w)\in K$.
This proves the proposition.
\end{proof}

By taking the Euler product of the identities in Definition~\ref{G}, we obtain the following proposition.

\begin{prop}\label{Fzeta}
For $\Re s>1$ and $\Re w>2k-1$, we have
\[
\mathcal{F}(s,w)=\zeta(s)\zeta(s+2w-4k+2)\zeta(s+3w-6k+3)\mathcal{G}_S(s,w),
\]
where $\mathcal{G}_S(s,w)$ is the Euler product defined in Proposition~\ref{PG}.
In particular, this gives a meromorphic continuation of $\mathcal F(s,w)$ to $R$.
\end{prop}

Using the above expression for $\mathcal{F}(s,w)$, we calculate the integral appearing in Proposition~\ref{Perron}:
\[
\int_{\kappa-iB^{8k}}^{\kappa+iB^{8k}}\int_{\lambda-iB^{8k}}^{\lambda+iB^{8k}}\frac{\mathcal{F}(s,w)X^{s+1}Y^{w+1}}{s(s+1)w(w+1)}\,\mathrm{d}w\mathrm{d}s=\int_{\kappa-iB^{8k}}^{\kappa+iB^{8k}}\frac{X^{s+1}}{s(s+1)}\left(\int_{\lambda-iB^{8k}}^{\lambda+iB^{8k}}\frac{\mathcal{F}(s,w)Y^{w+1}}{w(w+1)}\,\mathrm{d}w\right)\,\mathrm{d}s,
\]
where $\kappa=1+(\log B)^{-1}$ and $\lambda=2k-1+(\log B)^{-1}$.
First, we fix $s\in\mathbb{C}$ such that $\Re s=\kappa$ and $|\Im s|\leq B^{8k}$, and calculate the inner integral
\begin{align*}
&\int_{\lambda-iB^{8k}}^{\lambda+iB^{8k}}\frac{\mathcal{F}(s,w)Y^{w+1}}{w(w+1)}\,\mathrm{d}w
\\
&=\int_{\lambda-iB^{8k}}^{\lambda+iB^{8k}}\frac{\zeta(s)\zeta(s+2w-4k+2)\zeta(s+3w-6k+3)\mathcal{G}_S(s,w)Y^{w+1}}{w(w+1)}\,\mathrm{d}w.
\end{align*}
Let us consider the rectangle
\[
2k-1-\varepsilon\leq\Re w\leq\lambda,\qquad|\Im w|\leq B^{8k}.
\]
Since $B>e^{32}$, we have $(\log B)^{-1}<\varepsilon$.
Hence, the only possible poles of the integrand inside this rectangle are at $w=(4k-1-s)/2$, $w=(6k-2-s)/3$, and the corresponding residues are
\[
\frac{2\zeta(s)\zeta(\frac{3-s}{2})\mathcal{G}_S(s,\frac{4k-1-s}{2})}{(s-4k+1)(s-4k-1)}Y^{\frac{4k+1-s}{2}},\quad\frac{3\zeta(s)\zeta(\frac{s+2}{3})\mathcal{G}_S(s,\frac{6k-2-s}{3})}{(s-6k+2)(s-6k-1)}Y^{\frac{6k+1-s}{3}},
\]
respectively.

\begin{dfn}\label{IR}
We define functions $I_1$, $I_{\pm}$, $R_1$, and $R_2$ as follows:
\begin{enumerate}
\item $\displaystyle I_1(X,Y)=\int_{\kappa-iB^{8k}}^{\kappa+iB^{8k}}\int_{2k-1-\varepsilon-iB^{8k}}^{2k-1-\varepsilon+iB^{8k}}\frac{\mathcal{F}(s,w)X^{s+1}Y^{w+1}}{s(s+1)w(w+1)}\,\mathrm{d}w\mathrm{d}s$,
\item $\displaystyle I_{\pm}(X,Y)=\int_{\kappa-iB^{8k}}^{\kappa+iB^{8k}}\int_{2k-1-\varepsilon\pm iB^{8k}}^{\lambda\pm iB^{8k}}\frac{\mathcal{F}(s,w)X^{s+1}Y^{w+1}}{s(s+1)w(w+1)}\,\mathrm{d}w\mathrm{d}s$,
\item $\displaystyle R_1(X,Y)=\int_{\kappa-iB^{8k}}^{\kappa+iB^{8k}}\frac{2\zeta(s)\zeta(\frac{3-s}{2})\mathcal{G}_S(s,\frac{4k-1-s}{2})X^{s+1}Y^{\frac{4k+1-s}{2}}}{s(s+1)(s-4k+1)(s-4k-1)}\,\mathrm{d}s$,
\item $\displaystyle R_2(X,Y)=\int_{\kappa-iB^{8k}}^{\kappa+iB^{8k}}\frac{3\zeta(s)\zeta(\frac{s+2}{3})\mathcal{G}_S(s,\frac{6k-2-s}{3})X^{s+1}Y^{\frac{6k+1-s}{3}}}{s(s+1)(s-6k+2)(s-6k-1)}\,\mathrm{d}s$.
\end{enumerate} 
\end{dfn}

By Proposition~\ref{Perron} and Cauchy's residue theorem, we have the following proposition.
\begin{prop}\label{MIR}
We have
\[
M(X,Y)=\frac{I_1(X,Y)+I_{+}(X,Y)-I_{-}(X,Y)+2\pi iR_1(X,Y)+2\pi iR_2(X,Y)}{(2\pi i)^2}+O(1)
\]
as $B\to\infty$, uniformly for $B^{1-\varepsilon}/2\leq X\leq 2B$ and $B^{2-3\varepsilon}/2\leq Y\leq 2B^2$.
\end{prop}

Next, we calculate $E^{\pm}I_1, E^{\pm}I_{\pm}, E^{\pm}R_i, E_jI_1, E_jI_{\pm}$, and $E_jR_i$.
Recall that $E^{\pm}$ and $E_j$ are operators defined in Notation~\ref{nota}.
By Lemma~\ref{EMS} and Corollary~\ref{Cor}, these calculations yield asymptotic formulae for $S$ and $T$.
To estimate the above integrals, we recall the following bound for the Riemann zeta function.

\begin{lem}\label{zeta}
Let $\sigma, \tau\in\mathbb{R}$.
\begin{enumerate}
\item When $1<\sigma$, we have $|\zeta(\sigma+i\tau)|\leq\sigma(\sigma-1)^{-1}$.
\item When $1\leq\sigma$ and $|\tau|\geq2$, we have $|\zeta(\sigma+i\tau)|\ll\log(2+|\tau|)$.
\item When $1/2<\sigma<1-\varepsilon/2$, we have $|\zeta(\sigma+i\tau)|\ll(1+|\tau|)^{\frac{1-\sigma}{2}+\varepsilon}$.
\item When $1/2<\sigma<1$ and $1\leq|\tau|$, we have $|\zeta(\sigma+i\tau)|\ll(1+|\tau|)^{\frac{1-\sigma}{2}+\varepsilon}$.
\end{enumerate}
\end{lem}

\begin{proof}
The first inequality follows from a direct calculation.
The second inequality follows from the Euler summation formula.
The third and fourth inequalities follow from \cite[(5.1.4)]{Tit86}.
\end{proof}

\begin{prop}\label{Ipm}
We have
\[
I_{\pm}(X,Y)=O(1)
\]
as $B\to\infty$, uniformly for $B^{1-\varepsilon}/2\leq X\leq 2B$ and $B^{2-3\varepsilon}/2\leq Y\leq 2B^2$.
In particular, we have
\[
E^{\pm}I_{\pm}(X,Y)=O(1)
\]
uniformly for $B^{1-\varepsilon}/2\leq X\leq 2B$ and $B^{2-3\varepsilon}/2\leq Y\leq 2B^2$.
Moreover, for $j=1,2,3,4$, we have
\[
E_jI_{\pm}(B)=O((\log B)^2).
\]
\end{prop}

\begin{proof}
By Lemma~\ref{zeta}~(1), (2) and (4), for $(s,w)\in\mathbb{C}^2$ such that $\Re s=\kappa$, $|\Im s|\leq B^{8k}$, $2k-1-\varepsilon\leq\Re w\leq\lambda$, and $\Im w=\pm B^{8k}$, we have
\begin{align*}
|\zeta(s)|&\leq\left(1+\frac{1}{\log B}\right)\log B=\log B+1,
\\
|\zeta(s+2(w-2k+1))|&\ll(1+3B^{8k})^{(1-(1-2\varepsilon+(\log B)^{-1}))/2+\varepsilon}\ll B^{16\varepsilon k}=B^{\frac{1}{2}k},
\\
|\zeta(s+3(w-2k+1))|&\ll(1+4B^{8k})^{(1-(1-3\varepsilon+(\log B)^{-1}))/2+\varepsilon}\ll B^{20\varepsilon k}=B^{\frac{5}{8}k},
\end{align*}
and it follows that
\[
\mathcal{F}(s,w)\ll B^{\frac{9}{8}k}\log B.
\]
Hence, we obtain
\begin{align*}
I_{\pm}(X,Y)&=\int_{\kappa-iB^{8k}}^{\kappa+iB^{8k}}\int_{2k-1-\varepsilon\pm iB^{8k}}^{\lambda\pm iB^{8k}}\frac{\mathcal{F}(s,w)X^{s+1}Y^{w+1}}{s(s+1)w(w+1)}\,\mathrm{d}w\mathrm{d}s
\\
&\ll\frac{B^{\frac{9}{8}k}\log B\,X^{\kappa+1}Y^{\lambda+1}}{B^{16k}}\int_{\kappa-iB^{8k}}^{\kappa+iB^{8k}}\frac{\mathrm{d}s}{|s(s+1)|}
\\
&\ll\frac{B^{\frac{9}{8}k}\log B}{B^{4k-2}}=O(1).
\end{align*}
The remaining estimates follow from the definitions of $E^{\pm}$ and $E_j$, together with the bound $k_0=O((\log B)^2)$.\qedhere
\end{proof}

We use the following lemma to estimate $I_1, R_1$, and $R_2$.
This lemma is proved by an argument similar to that in \cite[Lemma~3]{dlB98}.

\begin{lem}\label{43}
We have
\[
|(X+H)^{\sigma+i\tau}-X^{\sigma+i\tau}|\ll X^{\sigma-\beta}(|\tau|+1)^\beta H^\beta
\]
uniformly for $0<H\leq X$, $|\sigma|\leq 3k$, $\beta\in[0,1]$, and $\tau\in\mathbb{R}$.
In particular, we have
\[
|(1+H)^s-1|\ll(|s|+1)H
\]
uniformly for $0<H<1$ and $|\Re s|\leq 3k$.
\end{lem}

\begin{prop}\label{I1}
We have
\[
E^{\pm}I_1(X,Y)\ll B^{4k-1-2\varepsilon+13\varepsilon^2}\log B
\]
as $B\to\infty$, uniformly for $B^{1-\varepsilon}\leq X\leq B$ and $B^{2-3\varepsilon}\leq Y\leq B^2$.
In particular, we have
\[
E_jI_1(B)\ll B^{4k-1-2\varepsilon+13\varepsilon^2}(\log B)^3.
\]
\end{prop}

\begin{proof}
Let $s=\sigma+i\tau, w=u+iv\in\mathbb{C}$ be such that $\sigma=\kappa$, $|\tau|\leq B^{8k}$, $u=2k-1-\varepsilon$, and $|v|\leq B^{8k}$.
By Lemma~\ref{zeta}, we have
\[
\mathcal{F}(s,w)\ll((1+|\tau|)^{5\varepsilon}+(1+|v|)^{5\varepsilon})\log B.
\]
By applying Lemma~\ref{43} with $\beta=\kappa-6\varepsilon$, we have
\begin{align*}
|(X\pm X^{1-\varepsilon})^{s+1}-X^{s+1}|&\ll X^{2-\varepsilon+6\varepsilon^2}(1+|\tau|)^{1-6\varepsilon}\ll B^{1+6\varepsilon^2}(1+|\tau|)^{1-6\varepsilon}X^{1-\varepsilon},
\\
|(Y\pm Y^{1-\varepsilon})^{w+1}-Y^{w+1}|&\ll B^{4k-2-2\varepsilon+12\varepsilon^2}(1+|v|)^{1-6\varepsilon}Y^{1-\varepsilon}.
\end{align*}
We note that $\kappa-6\varepsilon<1$ since $\varepsilon=32^{-1}$ and $B>e^{32}$.
On the other hand, we have
\begin{align*}
|(X\pm X^{1-\varepsilon})^{s+1}-X^{s+1}|&\ll B^{1+\varepsilon^2}(1+|\tau|)^{1-\varepsilon}X^{1-\varepsilon},
\\
|(Y\pm Y^{1-\varepsilon})^{w+1}-Y^{w+1}|&\ll B^{4k-2-2\varepsilon+2\varepsilon^2}(1+|v|)^{1-\varepsilon}Y^{1-\varepsilon}
\end{align*}
by applying Lemma~\ref{43} with $\beta=1-\varepsilon$.
Therefore, we obtain
\begin{align*}
&E^{\pm}I_1(X,Y)
\\
&=\int_{\kappa-iB^{8k}}^{\kappa+iB^{8k}}\int_{2k-1-\varepsilon-iB^{8k}}^{2k-1-\varepsilon+iB^{8k}}\frac{\mathcal{F}(s,w)((X\pm X^{1-\varepsilon})^{s+1}-X^{s+1})((Y\pm Y^{1-\varepsilon})^{w+1}-Y^{w+1})}{s(s+1)w(w+1)X^{1-\varepsilon}Y^{1-\varepsilon}}\,\mathrm{d}w\mathrm{d}s
\\
&\ll B^{1+6\varepsilon^2+4k-2-2\varepsilon+2\varepsilon^2}\log B\int_{\kappa-iB^{8k}}^{\kappa+iB^{8k}}\int_{2k-1-\varepsilon-iB^{8k}}^{2k-1-\varepsilon+iB^{8k}}\frac{(1+|\tau|)^{5\varepsilon}(1+|\tau|)^{1-6\varepsilon}(1+|v|)^{1-\varepsilon}}{|s(s+1)w(w+1)|}\,\mathrm{d}w\mathrm{d}s
\\
&+B^{4k-2-2\varepsilon+12\varepsilon^2+1+\varepsilon^2}\log B\int_{\kappa-iB^{8k}}^{\kappa+iB^{8k}}\int_{2k-1-\varepsilon-iB^{8k}}^{2k-1-\varepsilon+iB^{8k}}\frac{(1+|v|)^{5\varepsilon}(1+|v|)^{1-6\varepsilon}(1+|\tau|)^{1-\varepsilon}}{|s(s+1)w(w+1)|}\,\mathrm{d}w\mathrm{d}s
\\
&\ll B^{4k-1-2\varepsilon+13\varepsilon^2}\log B.
\end{align*}
The latter part follows from $k_0=O(\phi(B)^{-1}(\log B)^2)=O((\log B)^2)$.
\end{proof}

We next evaluate the integral
\[
R_1(X,Y)=2\int_{\kappa-iB^{8k}}^{\kappa+iB^{8k}}\frac{\zeta(s)\zeta(\frac{3-s}{2})\mathcal{G}_S(s,\frac{4k-1-s}{2})X^{s+1}Y^{\frac{4k+1-s}{2}}}{s(s+1)(s-4k+1)(s-4k-1)}\,\mathrm{d}s.
\]

\begin{prop}\label{ER1}
$E^{\pm}R_1(B,B^2)=O(B^{4k-1})\quad(B\to\infty)$.
\end{prop}

\begin{proof}
Since $B>e^{32}$, we have $\kappa<1+\varepsilon$.
We consider the rectangle
\[
\kappa\leq\Re s\leq1+\varepsilon,\qquad|\Im s|\leq B^{8k}.
\]
We have
\[
\int_{\kappa\pm iB^{8k}}^{1+\varepsilon\pm iB^{8k}}\frac{\zeta(s)\zeta(\frac{3-s}{2})\mathcal{G}_S(s,\frac{4k-1-s}{2})X^{s+1}Y^{\frac{4k+1-s}{2}}}{s(s+1)(s-4k+1)(s-4k-1)}\,\mathrm{d}s=O(1)
\]
by an argument similar to that in the proof of Proposition~\ref{Ipm}.
Hence, we have
\[
R_1(X,Y)=2\int_{1+\varepsilon-iB^{8k}}^{1+\varepsilon+iB^{8k}}\frac{\zeta(s)\zeta(\frac{3-s}{2})\mathcal{G}_S(s,\frac{4k-1-s}{2})X^{s+1}Y^{\frac{4k+1-s}{2}}}{s(s+1)(s-4k+1)(s-4k-1)}\,\mathrm{d}s+O(1)
\]
by Cauchy's integral theorem.
By Lemma~\ref{zeta}, for $s\in\mathbb{C}$ such that $\Re s=1+\varepsilon$ and $|\Im s|\leq B^{8k}$, we have
\[
\zeta(s)\zeta\left(\frac{3-s}{2}\right)\mathcal{G}_S\left(s,\frac{4k-1-s}{2}\right)\ll(1+|\tau|)^{\frac{5}{4}\varepsilon}.
\]
By applying Lemma~\ref{43} with $\beta=1$, we have
\begin{align*}
|(B\pm B^{1-\varepsilon})^{s+1}-B^{s+1}|&\ll B^2|s+1|,
\\
|(B^2\pm B^{2-2\varepsilon})^{\frac{4k+1-s}{2}}-B^{4k+1-s}|&\ll B^{4k-3\varepsilon}|4k+1-s|.
\end{align*}
Therefore, we obtain
\begin{align*}
&E^{\pm}R_1(B,B^2)
\\
&=2\int_{1+\varepsilon-iB^{8k}}^{1+\varepsilon+iB^{8k}}\frac{\zeta(s)\zeta\left(\frac{3-s}{2}\right)\mathcal{G}_S\left(s,\frac{4k-1-s}{2}\right)}{s(s+1)(s-4k-1)(s-4k+1)B^{3-3\varepsilon}}
\\
&\quad\quad\quad\quad\times((B\pm B^{1-\varepsilon})^{s+1}-B^{s+1})((B^2\pm B^{2-2\varepsilon})^{\frac{4k+1-s}{2}}-B^{4k+1-s})\,\mathrm{d}s+O(1)
\\
&\ll B^{4k-1}\int_{1+\varepsilon-iB^{8k}}^{1+\varepsilon+iB^{8k}}\frac{(1+|\tau|)^{\frac{5\varepsilon}{4}}}{|s(s-4k+1)|}\,\mathrm{d}s
\\
&\ll B^{4k-1}.\qedhere
\end{align*}
\end{proof}

\begin{prop}\label{EjR1}
We have
\begin{align*}
E_1R_1(B)-E_2R_1(B)&=O(B^{4k-1}\phi(B))\quad(B\to\infty),
\\
E_4R_1(B)-E_3R_1(B)&=O(B^{4k-1}\phi(B))\quad(B\to\infty).
\end{align*}
\end{prop}

\begin{proof}
Since the proofs are analogous, we prove only
\[
E_1R_1(B)-E_2R_1(B)=O(B^{4k-1}).
\]
We set
\[
R_1'(X,Y)=\int_{1+\varepsilon-iB^{8k}}^{1+\varepsilon+iB^{8k}}\frac{2\zeta(s)\zeta(\frac{3-s}{2})\mathcal{G}_S(s,\frac{4k-1-s}{2})X^{s+1}Y^{\frac{4k+1-s}{2}}}{s(s+1)(s-4k+1)(s-4k-1)}\,\mathrm{d}s.
\]
As in the proof of Proposition~\ref{ER1}, we have $R_1(X,Y)=R_1'(X,Y)+O(1)$ uniformly for $B^{1-\varepsilon}/2\leq X\leq 2B$ and $B^{2-3\varepsilon}/2\leq Y\leq 2B^2$.
Hence, it suffices to show that
\[
E_1R_1'(B)-E_2R_1'(B)=O(B^{4k-1}).
\]
We have
\begin{align*}
&E^{\pm}R_1'(X,Y)
\\
&=\frac{R_1'(X\pm X^{1-\varepsilon},Y\pm Y^{1-\varepsilon})+R_1'(X,Y)-R_1'(X\pm X^{1-\varepsilon},Y)-R_1'(X,Y\pm Y^{1-\varepsilon})}{X^{1-\varepsilon}Y^{1-\varepsilon}}
\\
&=\int_0^1\int_0^1\frac{\partial^2}{\partial x\partial y}R_1'(X\pm uX^{1-\varepsilon},Y\pm vY^{1-\varepsilon})\,\mathrm{d}u\mathrm{d}v
\\
&=-\int_0^1\int_0^1\int_{1+\varepsilon-iB^{8k}}^{1+\varepsilon+iB^{8k}}\frac{\zeta(s)\zeta(\frac{3-s}{2})\mathcal{G}_S(s,\frac{4k-1-s}{2})(X\pm uX^{1-\varepsilon})^s(Y\pm vY^{1-\varepsilon})^{\frac{4k-1-s}{2}}}{s(s-4k+1)}\,\mathrm{d}s\mathrm{d}u\mathrm{d}v
\\
&=-\int_{1+\varepsilon-iB^{8k}}^{1+\varepsilon+iB^{8k}}\frac{\zeta(s)\zeta(\frac{3-s}{2})\mathcal{G}_S(s,\frac{4k-1-s}{2})}{s(s-4k+1)}X^sY^{\frac{4k-1-s}{2}}\left(\int_0^1(1\pm uX^{-\varepsilon})^s\,\mathrm{d}u\right)\left(\int_0^1(1\pm vY^{-\varepsilon})^{\frac{4k-1-s}{2}}\,\mathrm{d}v\right)\,\mathrm{d}s.
\end{align*}
By Lemma~\ref{zeta} and Lemma~\ref{43}, we have
\begin{align*}
\frac{\zeta(s)\zeta(\frac{3-s}{2})\mathcal{G}_S(s,\frac{4k-1-s}{2})}{s(s-4k+1)}&\ll(1+|\tau|)^{-2+\frac{5}{4}\varepsilon},
\\
\left(\int_0^1(1\pm uX^{-\varepsilon})^s\,\mathrm{d}u\right)\left(\int_0^1(1\pm vY^{-\varepsilon})^{\frac{4k-1-s}{2}}\,\mathrm{d}v\right)-1&\ll(1+|\tau|)^{\frac{1}{2}}(X^{-\frac{\varepsilon}{2}}+Y^{-\frac{\varepsilon}{2}}).
\end{align*}
Hence, we have
\begin{align*}
&E^{\pm}R_1'(X,Y)-\frac{\partial^2}{\partial x\partial y}R_1'(X,Y)
\\
&\ll\int_{1+\varepsilon-iB^{8k}}^{1+\varepsilon+iB^{8k}}(1+|\tau|)^{-2+\frac{5}{4}\varepsilon}B^{4k-1}(1+|\tau|)^{\frac{1}{2}}(X^{-\frac{\varepsilon}{2}}+Y^{-\frac{\varepsilon}{2}})\,\mathrm{d}s
\\
&=B^{4k-1}(X^{-\frac{\varepsilon}{2}}+Y^{-\frac{\varepsilon}{2}})\int_{1+\varepsilon-iB^{8k}}^{1+\varepsilon+iB^{8k}}(1+|\tau|)^{-\frac{3}{2}+\frac{5}{4}\varepsilon}\,\mathrm{d}s
\\
&=O(B^{4k-1-\frac{\varepsilon(1-\varepsilon)}{2}}).
\end{align*}
It follows that
\[
E^{\pm}R_1'(X,Y)=-\int_{1+\varepsilon-iB^{8k}}^{1+\varepsilon+iB^{8k}}\frac{\zeta(s)\zeta(\frac{3-s}{2})\mathcal{G}_S(s,\frac{4k-1-s}{2})}{s(s-4k+1)}X^sY^{\frac{4k-1-s}{2}}\,\mathrm{d}s+O(B^{4k-1-\frac{\varepsilon(1-\varepsilon)}{2}}).
\]
Hence, for all $1\leq l\leq k_0-1$, we have
\begin{align*}
&E^{+}R_1'(\delta^{l-1}B,\delta^{3(l-1)}B^2)-E^{-}R_1'(\delta^lB,\delta^{3(l-1)}B^2)
\\
&=-\int_{1+\varepsilon-iB^{8k}}^{1+\varepsilon+iB^{8k}}\frac{\zeta(s)\zeta(\frac{3-s}{2})\mathcal{G}_S(s,\frac{4k-1-s}{2})}{s(s-4k+1)}(\delta^{l-1}B)^s(\delta^{3(l-1)}B^2)^{\frac{4k-1-s}{2}}\,\mathrm{d}s
\\
&+\int_{1+\varepsilon-iB^{8k}}^{1+\varepsilon+iB^{8k}}\frac{\zeta(s)\zeta(\frac{3-s}{2})\mathcal{G}_S(s,\frac{4k-1-s}{2})}{s(s-4k+1)}(\delta^lB)^s(\delta^{3(l-1)}B^2)^{\frac{4k-1-s}{2}}\,\mathrm{d}s+O(B^{4k-1-\frac{\varepsilon(1-\varepsilon)}{2}})
\\
&=-B^{4k-1}\int_{1+\varepsilon-iB^{8k}}^{1+\varepsilon+iB^{8k}}\frac{\zeta(s)\zeta(\frac{3-s}{2})\mathcal{G}_S(s,\frac{4k-1-s}{2})}{s(s-4k+1)}\delta^{(l-1)(6k-\frac{3}{2}-\frac{1}{2}s)}(1-\delta^s)\,\mathrm{d}s+O(B^{4k-1-\frac{\varepsilon(1-\varepsilon)}{2}}).
\end{align*}
It follows that
\begin{align*}
&E_1R_1'(B)-E_2R_1'(B)
\\
&=\sum_{l=1}^{k_0-1}(E^{+}R_1'(\delta^{l-1}B,\delta^{3(l-1)}B^2)-E^{-}R_1'(\delta^lB,\delta^{3(l-1)}B^2))
\\
&=-B^{4k-1}\int_{1+\varepsilon-iB^{8k}}^{1+\varepsilon+iB^{8k}}\frac{\zeta(s)\zeta(\frac{3-s}{2})\mathcal{G}_S(s,\frac{4k-1-s}{2})}{s(s-4k+1)}\frac{(1-\delta^s)(1-\delta^{(6k-\frac{3}{2}-\frac{s}{2})(k_0-1)})}{1-\delta^{6k-\frac{3}{2}-\frac{s}{2}}}\,\mathrm{d}s+O(k_0B^{4k-1-\frac{\varepsilon(1-\varepsilon)}{2}}).
\end{align*}
The proposition follows from the fact that
\begin{align*}
\frac{\zeta(s)\zeta(\frac{3-s}{2})\mathcal{G}_S(s,\frac{4k-1-s}{2})}{s(s-4k+1)}&\ll(1+|\tau|)^{-2+\frac{5}{4}\varepsilon},
\\
\left|\frac{(1-\delta^s)(1-\delta^{(6k-\frac{3}{2}-\frac{s}{2})(k_0-1)})}{1-\delta^{6k-\frac{3}{2}-\frac{s}{2}}}\right|&\leq2(12k-1),
\\
k_0&=O((\log B)^2).\qedhere
\end{align*}
\end{proof}

Next, we calculate
\[
R_2(X,Y)=\int_{\kappa-iB^{8k}}^{\kappa+iB^{8k}}\frac{3\zeta(s)\zeta(\frac{s+2}{3})\mathcal{G}_S(s,\frac{6k-2-s}{3})X^{s+1}Y^{\frac{6k+1-s}{3}}}{s(s+1)(s-6k+2)(s-6k-1)}\,\mathrm{d}s.
\]
This gives the main term of the asymptotic formula for $M(X,Y)$.

\begin{prop}\label{R2}
There exist a real number $c'$ and a function $R_3\colon\mathbb{R}_{>0}^2\to\mathbb{C}$ such that the following estimates hold as $B\to\infty$, uniformly for $B^{1-\varepsilon}/2\leq X\leq 2B$ and $B^{2-3\varepsilon}/2\leq Y\leq 2B^2$:
\begin{align*}
R_2(X,Y)&=2\pi i\left(\frac{\mathcal{G}_S(1,2k-1)}{4k(2k-1)}X^2Y^{2k}\left(\log X-\frac{1}{3}\log Y\right)+c'X^2Y^{2k}\right)+R_3(X,Y)+O(1),
\\
E^{\pm}R_3(X,Y)&\ll B^{4k-1-\varepsilon/3}.
\end{align*}
\end{prop}

\begin{proof}
We set
\begin{align*}
c'&=\left.\left(\frac{9\mathcal{G}_S(s,\frac{6k-2-s}{3})}{s(s+1)(s-6k+2)(s-6k-1)}\right)'\right|_{s=1}+\frac{\gamma\mathcal{G}_S(1,2k-1)}{3k(2k-1)},
\\
R_3(X,Y)&=\int_{1-\varepsilon-iB^{8k}}^{1-\varepsilon+iB^{8k}}\frac{3\zeta(s)\zeta(\frac{s+2}{3})\mathcal{G}_S(s,\frac{6k-2-s}{3})X^{s+1}Y^{\frac{6k+1-s}{3}}}{s(s+1)(s-6k+2)(s-6k-1)}\,\mathrm{d}s,
\end{align*}
where $\gamma$ is Euler's constant.
Let us consider the rectangle
\[
1-\varepsilon\leq\Re s\leq\kappa,\qquad|\Im s|\leq B^{8k}.
\]
The integrand
\[
\frac{3\zeta(s)\zeta(\frac{s+2}{3})\mathcal{G}_S(s,\frac{6k-2-s}{3})X^{s+1}Y^{\frac{6k+1-s}{3}}}{s(s+1)(s-6k+2)(s-6k-1)}
\]
has a pole at $s=1$ of order 2, and its residue is
\[
\frac{\mathcal{G}_S(1,2k-1)}{4k(2k-1)}X^2Y^{2k}\left(\log X-\frac{1}{3}\log Y\right)+c'X^2Y^{2k}.
\]
We note that $\mathcal{G}_S(1,2k-1)\neq0$ will be shown in Proposition~\ref{Gp}.
By Cauchy's residue theorem, we obtain
\begin{align*}
R_2(X,Y)&=2\pi i\left(\frac{\mathcal{G}_S(1,2k-1)}{4k(2k-1)}X^2Y^{2k}\left(\log X-\frac{1}{3}\log Y\right)+c'X^2Y^{2k}\right)+R_3(X,Y)
\\
&+\int_{1-\varepsilon+iB^{8k}}^{\kappa+iB^{8k}}\frac{3\zeta(s)\zeta(\frac{s+2}{3})\mathcal{G}_S(s,\frac{6k-2-s}{3})X^{s+1}Y^{\frac{6k+1-s}{3}}}{s(s+1)(s-6k+2)(s-6k-1)}\,\mathrm{d}s
\\
&-\int_{1-\varepsilon-iB^{8k}}^{\kappa-iB^{8k}}\frac{3\zeta(s)\zeta(\frac{s+2}{3})\mathcal{G}_S(s,\frac{6k-2-s}{3})X^{s+1}Y^{\frac{6k+1-s}{3}}}{s(s+1)(s-6k+2)(s-6k-1)}\,\mathrm{d}s.
\end{align*}
As in the proof of Proposition~\ref{Ipm}, we obtain
\[
\int_{1-\varepsilon\pm iB^{8k}}^{\kappa\pm iB^{8k}}\frac{3\zeta(s)\zeta(\frac{s+2}{3})\mathcal{G}_S(s,\frac{6k-2-s}{3})X^{s+1}Y^{\frac{6k+1-s}{3}}}{s(s+1)(s-6k+2)(s-6k-1)}\,\mathrm{d}s=O(1).
\]
By Lemma~\ref{zeta} and Lemma~\ref{43}, for $\Re s=1-\varepsilon$, we have
\begin{align*}
\zeta(s)\zeta\left(\frac{s+2}{3}\right)\mathcal{G}_S\left(s,\frac{6k-2-s}{3}\right)&\ll(|\tau|+1)^{\frac{8\varepsilon}{3}},
\\
((X\pm X^{1-\varepsilon})^{s+1}-X^{s+1})((Y\pm Y^{1-\varepsilon})^{\frac{6k+1-s}{3}}-Y^{\frac{6k+1-s}{3}})&\ll X^{1-\varepsilon}Y^{1-\varepsilon}|(s+1)(s-6k-1)|B^{4k-1-\frac{\varepsilon}{3}},
\end{align*}
where $\tau=\Im s$.
Therefore, we obtain
\begin{align*}
&E^{\pm}R_3(X,Y)
\\
&=\int_{1-\varepsilon-iB^{8k}}^{1-\varepsilon+iB^{8k}}\frac{3\zeta(s)\zeta(\frac{s+2}{3})\mathcal{G}_S(s,\frac{6k-2-s}{3})((X\pm X^{1-\varepsilon})^{s+1}-X^{s+1})((Y\pm Y^{1-\varepsilon})^{\frac{6k+1-s}{3}}-Y^{\frac{6k+1-s}{3}})}{s(s+1)(s-6k+2)(s-6k-1)X^{1-\varepsilon}Y^{1-\varepsilon}}\,\mathrm{d}s
\\
&\ll\int_{1-\varepsilon-iB^{8k}}^{1-\varepsilon+iB^{8k}}\frac{3(|\tau|+1)^{\frac{8\varepsilon}{3}}B^{4k-1-\frac{\varepsilon}{3}}}{|s(s-6k+2)|}\,\mathrm{d}s
\\
&\ll B^{4k-1-\frac{\varepsilon}{3}}.\qedhere
\end{align*}
\end{proof}

Finally, we establish asymptotic formulae for $S(B,B^2)$ and $T(B)$.
We define a function $g\colon\mathbb{R}_{>0}^2\to\mathbb{R}$ by
\[
g(X,Y)=\frac{\mathcal{G}_S(1,2k-1)}{4k(2k-1)}X^2Y^{2k}\left(\log X-\frac{1}{3}\log Y\right)+c'X^2Y^{2k},
\]
where $c'$ is the constant defined in the proof of Proposition~\ref{R2}.
The functions $E^{\pm}g$ and $E_jg$ can be evaluated by the following lemma, proved in \cite[Lemma~3]{dlB98}.

\begin{lem}\label{3}
Let $f\colon\mathbb{R}_{>0}^2\to\mathbb{R}$ be a $C^3$ function.
For $0<H<X$ and $0<J<Y$, we have
\begin{align*}
\mathscr{D}f(X,X+H;Y,Y+J)&=HJ\left(\frac{\partial^2f}{\partial x\partial y}(X,Y)+O(R_{+}f(X,H,Y,J))\right),
\\
\mathscr{D}f(X-H,X;Y-J,Y)&=HJ\left(\frac{\partial^2f}{\partial x\partial y}(X,Y)+O(R_{-}f(X,H,Y,J))\right),
\end{align*}
where 
\begin{align*}
R_{+}f(X,H,Y,J)&=H\max_{\substack{X\leq x\leq X+H,\\ Y\leq y\leq Y+J}}\left|\frac{\partial^3f}{\partial x^2\partial y}(x,y)\right|+J\max_{\substack{X\leq x\leq X+H,\\Y\leq y\leq Y+J}}\left|\frac{\partial^3f}{\partial x\partial y^2}(x,y)\right|,
\\
R_{-}f(X,H,Y,J)&=H\max_{\substack{X-H\leq x\leq X,\\ Y-J\leq y\leq Y}}\left|\frac{\partial^3f}{\partial x^2\partial y}(x,y)\right|+J\max_{\substack{X-H\leq x\leq X,\\Y-J\leq y\leq Y}}\left|\frac{\partial^3f}{\partial x\partial y^2}(x,y)\right|.
\end{align*}
\end{lem}

By applying Lemma~\ref{3} to $g$, we obtain the following result.

\begin{co}\label{Epmg}
There is $c''\in\mathbb{R}$ such that
\[
E^{\pm}g(X,Y)=\frac{\mathcal{G}_S(1,2k-1)}{2k-1}XY^{2k-1}\left(\log X-\frac{1}{3}\log Y\right)+c''XY^{2k-1}+O(X^{1-\frac{\varepsilon}{6}}Y^{2k-1})
\]
uniformly for $B^{1-\varepsilon}\leq X\leq B$ and $B^{2-3\varepsilon}\leq Y\leq B^2$.
In particular, we have
\[
E^{\pm}g(B,B^2)=\frac{\mathcal{G}_S(1,2k-1)}{3(2k-1)}B^{4k-1}\log B+O(B^{4k-1})\quad(B\to\infty).
\]
\end{co}

\begin{proof}
By Lemma~\ref{3}, we have
\begin{align*}
E^{\pm}g(X,Y)&=\frac{\mathscr{D}g(X,X\pm X^{1-\varepsilon};Y,Y\pm Y^{1-\varepsilon})}{X^{1-\varepsilon}Y^{1-\varepsilon}}
\\
&=\frac{\partial^2g}{\partial x\partial y}(X,Y)+O(R_{\pm}g(X,X^{1-\varepsilon},Y,Y^{1-\varepsilon})).
\end{align*}
The statement follows from the fact that
\[
\frac{\partial^2g}{\partial x\partial y}(X,Y)=\frac{\mathcal{G}_S(1,2k-1)}{2k-1}XY^{2k-1}\left(\log X-\frac{1}{3}\log Y\right)+c''XY^{2k-1},
\]
where
\[
c''=2k\left(\frac{\mathcal{G}_S(1,2k-1)}{4k(2k-1)}+2c'\right)-\frac{\mathcal{G}_S(1,2k-1)}{6k(2k-1)}.
\]
\end{proof}

We establish an asymptotic formula for $S(B,B^2)$ using the operators $E^{\pm}$.

\begin{prop}\label{SS}
The asymptotic formula
\[
S(B,B^2)=\frac{\mathcal{G}_S(1,2k-1)}{3(2k-1)}B^{4k-1}\log B+O(B^{4k-1})\quad(B\to\infty)
\]
holds.
\end{prop}

\begin{proof}
By Lemma~\ref{EMS}, Proposition~\ref{MIR}, and Proposition~\ref{R2}, we have
\begin{align*}
&S(B,B^2)
\\
&\leq E^{+}M(B,B^2)
\\
&=\frac{E^{+}I_1(B,B^2)+E^{+}I_{+}(B,B^2)-E^{+}I_{-}(B,B^2)+2\pi iE^{+}R_1(B,B^2)+2\pi iE^{+}R_2(B,B^2)}{(2\pi i)^2}+O(1)
\\
&=\frac{E^{+}I_1(B,B^2)+E^{+}I_{+}(B,B^2)-E^{+}I_{-}(B,B^2)+2\pi iE^{+}R_1(B,B^2)+2\pi iE^{+}R_3(B,B^2)}{(2\pi i)^2}
\\
&\quad+E^{+}g(B,B^2)+O(1).
\end{align*}
By Proposition~\ref{Ipm}, Proposition~\ref{I1}, Proposition~\ref{ER1}, and Proposition~\ref{R2}, we have
\begin{align*}
&E^{+}I_1(B,B^2)+E^{+}I_{+}(B,B^2)-E^{+}I_{-}(B,B^2)+2\pi iE^{+}R_1(B,B^2)+2\pi iE^{+}R_3(B,B^2)
\\
&=O(1+B^{4k-1-2\varepsilon+13\varepsilon^2}\log B+B^{4k-1}+B^{4k-1-\frac{\varepsilon}{3}})=O(B^{4k-1}).
\end{align*}
Hence, by Corollary~\ref{Epmg}, we obtain
\[
S(B,B^2)\leq\frac{\mathcal{G}_S(1,2k-1)}{3(2k-1)}B^{4k-1}\log B+O(B^{4k-1}).
\]
We can show the reverse inequality by using the operator $E^{-}$ instead of $E^{+}$.
\end{proof}

The functions $E_jg$ give an asymptotic formula for $T(B)$.

\begin{prop}
The asymptotic formula
\[
T(B)=\frac{\mathcal{G}_S(1,2k-1)}{6(2k-1)(3k-1)}B^{4k-1}\log B+O(B^{4k-1}\phi(B))\quad(B\to\infty)
\]
holds.
\end{prop}

\begin{proof}
By Corollary~\ref{Cor}, Proposition~\ref{MIR} and Proposition~\ref{R2}, we have
\begin{align*}
&\frac{E_4I_1(B)+E_4I_+(B)-E_4I_-(B)+2\pi iE_4R_1(B)+2\pi i E_4R_3(B)}{(2\pi i)^2}+E_4g(B)
\\
&-\frac{E_3I_1(B)+E_3I_+(B)-E_3I_-(B)+2\pi iE_3R_1(B)+2\pi i E_3R_3(B)}{(2\pi i)^2}-E_3g(B)+O(B^{4k-1-4\varepsilon}(\log B)^{30})
\\
&\leq T(B)
\\
\leq&\frac{E_1I_1(B)+E_1I_+(B)-E_1I_-(B)+2\pi iE_1R_1(B)+2\pi i E_1R_3(B)}{(2\pi i)^2}+E_1g(B)
\\
&-\frac{E_2I_1(B)+E_2I_+(B)-E_2I_-(B)+2\pi iE_2R_1(B)+2\pi i E_2R_3(B)}{(2\pi i)^2}-E_2g(B)+O(B^{4k-1-4\varepsilon}(\log B)^{30}).
\end{align*}
By Proposition~\ref{Ipm}, Proposition~\ref{I1}, Proposition~\ref{EjR1}, and Proposition~\ref{R2}, we have
\[
E_4g(B)-E_3g(B)+O(B^{4k-1}\phi(B))\leq T(B)\leq E_1g(B)-E_2g(B)+O(B^{4k-1}\phi(B)).
\]
We set
\[
c=\frac{\mathcal{G}_S(1,2k-1)}{2k-1}.
\]
Then we have
\[
E^{\pm}g(X,Y)=cXY^{2k-1}\left(\log X-\frac{1}{3}\log Y\right)+c''XY^{2k-1}+O\left(X^{1-\frac{\varepsilon}{6}}Y^{2k-1}\right)
\]
and it follows that
\begin{align*}
&E^{+}g(\delta^{l-1}B,\delta^{3(l-1)}B^2)
\\
&=c\delta^{l-1}B(\delta^{3(l-1)}B^2)^{2k-1}\left(\log\delta^{l-1}B-\frac{1}{3}\log\delta^{3(l-1)}B^2\right)
\\
&\quad+c''(\delta^{l-1}B)(\delta^{3(l-1)}B^2)^{2k-1}+O\left((\delta^{l-1}B)^{1-\frac{\varepsilon}{6}}(\delta^{3(l-1)}B^2)^{2k-1}\right)
\\
&=\frac{c}{3}\delta^{(6k-2)(l-1)}B^{4k-1}\log B+c''\delta^{(6k-2)(l-1)}B^{4k-1}+O(B^{4k-1-\frac{\varepsilon}{6}}),
\\
&E^{-}g(\delta^lB,\delta^{3(l-1)}B^2)
\\
&=c\delta^lB(\delta^{3(l-1)}B^2)^{2k-1}\left(\log\delta^lB-\frac{1}{3}\log\delta^{3(l-1)}B^2\right)
\\
&\quad+c''(\delta^lB)(\delta^{3(l-1)}B^2)^{2k-1}+O\left((\delta^lB)^{1-\frac{\varepsilon}{6}}(\delta^{3(l-1)}B^2)^{2k-1}\right)
\\
&=\frac{c}{3}\delta^{(6k-2)(l-1)+1}B^{4k-1}\log B+c''\delta^{(6k-2)(l-1)+1}B^{4k-1}+c\delta^{(6k-2)(l-1)+1}B^{4k-1}\log\delta+O(B^{4k-1-\frac{\varepsilon}{6}})
\end{align*}
for all $1\leq l\leq k_0-1$.
Hence, we have
\begin{align*}
&E^{+}g(\delta^{l-1}B,\delta^{3(l-1)}B^2)-E^{-}g(\delta^lB,\delta^{3(l-1)}B^2)
\\
&=\delta^{(6k-2)(l-1)}B^{4k-1}\left(\frac{c}{3}\log B+c''\right)(1-\delta)-cB^{4k-1}\delta^{(6k-2)(l-1)+1}\log\delta+O(B^{4k-1-\frac{\varepsilon}{6}})
\end{align*}
for all $1\leq l\leq k_0-1$.
Note that
\[
\frac{1-\delta}{1-\delta^{6k-2}}=\frac{1}{6k-2}+O\left(\frac{\phi(B)}{\log B}\right),\quad\frac{\log\delta}{1-\delta^{6k-2}}=O(1),\quad k_0=O((\log B)^2).
\]
Therefore, we obtain
\begin{align*}
&E_1g(B)-E_2g(B)
\\
&=\sum_{l=1}^{k_0-1}\left(E^{+}g(\delta^{l-1}B,\delta^{3(l-1)}B^2)-E^{-}g(\delta^lB,\delta^{3(l-1)}B^2)\right)
\\
&=\sum_{l=1}^{k_0-1}\left(\delta^{(6k-2)(l-1)}B^{4k-1}\left(\frac{c}{3}\log B+c''\right)(1-\delta)-cB^{4k-1}\delta^{(6k-2)(l-1)+1}\log\delta+O(B^{4k-1-\frac{\varepsilon}{6}})\right)
\\
&=B^{4k-1}\left(\frac{c}{3}\log B+c''\right)(1-\delta)\sum_{l=1}^{k_0-1}\delta^{(6k-2)(l-1)}-cB^{4k-1}\log\delta\sum_{l=1}^{k_0-1}\delta^{(6k-2)(l-1)+1}+O(B^{4k-1-\frac{\varepsilon}{6}}k_0)
\\
&=B^{4k-1}\left(\frac{c}{3}\log B+c''\right)(1-\delta^{(6k-2)(k_0-1)})\frac{1-\delta}{1-\delta^{6k-2}}
\\
&\quad-cB^{4k-1}\delta(1-\delta^{(6k-2)(k_0-1)})\frac{\log\delta}{1-\delta^{6k-2}}+O(B^{4k-1-\frac{\varepsilon}{6}}(\log B)^2)
\\
&=B^{4k-1}\left(\frac{c}{3}\log B+c''\right)(1-\delta^{(6k-2)(k_0-1)})\left(\frac{1}{6k-2}+O\left(\frac{\phi(B)}{\log B}\right)\right)+O(B^{4k-1})
\\
&=\frac{c}{6(3k-1)}B^{4k-1}\log B+O(B^{4k-1}\phi(B))
\\
&=\frac{\mathcal{G}_S(1,2k-1)}{6(2k-1)(3k-1)}B^{4k-1}\log B+O(B^{4k-1}\phi(B)).
\end{align*}
Similarly, we have
\[
E_4g(B)-E_3g(B)=\frac{\mathcal{G}_S(1,2k-1)}{6(2k-1)(3k-1)}B^{4k-1}\log B+O(B^{4k-1}\phi(B)).
\]
This completes the proof.
\end{proof}

Since
\[
N^{*}(B)=\frac{8k}{(4^k-1)|B_{2k}|}(S(B,B^2)-T(B))+O(B^{4k-3/2}(\log B)^3)\quad(B\to\infty),
\]
we obtain the following statement.

\begin{prop}\label{Nas}
The asymptotic formula
\[
N^{*}(B)=\frac{4k\mathcal{G}_S(1,2k-1)}{(3k-1)(4^k-1)|B_{2k}|}B^{4k-1}\log B+O(B^{4k-1}\phi(B))\quad(B\to\infty)
\]
holds.
\end{prop}

Finally, we establish an asymptotic formula for $N(B)$.

\begin{prop}\label{N}
For all functions $\phi\colon\mathbb{R}_{>0}\to\mathbb{R}_{>0}$ satisfying Assumption~\ref{3ass}, the asymptotic formula
\[
N(B)=\frac{2k\mathcal{G}_S(1,2k-1)}{(3k-1)(4^k-1)|B_{2k}|\zeta(4k-1)}B^{4k-1}\log B+O(B^{4k-1}\phi(B))\quad(B\to\infty)
\]
holds.
\end{prop}

\begin{proof}
By the inclusion-exclusion principle, we have
\begin{align*}
&N(B)
\\
&=\frac{1}{2}\sum_{d\leq B}\mu(d)N^{*}\left(\frac{B}{d}\right)
\\
&=\frac{2k\mathcal{G}_S(1,2k-1)}{(3k-1)(4^k-1)|B_{2k}|}\sum_{d\leq B}\mu(d)\left(\left(\frac{B}{d}\right)^{4k-1}(\log B-\log d)\right)+O\left(\sum_{d\leq B}|\mu(d)|\left(\frac{B}{d}\right)^{4k-1}\phi\left(\frac{B}{d}\right)\right)
\\
&=\frac{2k\mathcal{G}_S(1,2k-1)}{(3k-1)(4^k-1)|B_{2k}|}\left(\sum_{d\leq B}\frac{\mu(d)}{d^{4k-1}}\right)B^{4k-1}\log B
\\
&\qquad-\frac{2k\mathcal{G}_S(1,2k-1)}{(3k-1)(4^k-1)|B_{2k}|}\left(\sum_{d\leq B}\frac{\mu(d)\log d}{d^{4k-1}}\right)B^{4k-1}+O\left(B^{4k-1}\phi(B)\sum_{d\leq B}\frac{1}{d^{4k-1}}\right).
\end{align*}
The proposition follows from
\[
\sum_{d\leq B}\frac{\mu(d)\log d}{d^{4k-1}}=O(1),\quad\sum_{d\leq B}\frac{\mu(d)}{d^{4k-1}}=\frac{1}{\zeta(4k-1)}+O\left(\frac{1}{B}\right),\quad\sum_{d\leq B}\frac{1}{d^{4k-1}}=O(1).\qedhere
\]
\end{proof}

We introduce the following lemma to eliminate $\phi(B)$ from the error term.

\begin{lem}\label{f}
Let $f\colon\mathbb{R}_{>0}\to\mathbb{R}$ be a function. 
We assume that the asymptotic formula
\[
f(B)=O(\phi(B))\quad(B\to\infty)
\]
holds for all functions $\phi:\mathbb{R}_{>0}\to\mathbb{R}_{>0}$ satisfying Assumption~\ref{3ass}.
Then $f(x)=O(1)$.
\end{lem}

\begin{proof}
Assume that $f(x)$ is not $O(1)$.
Then there exists a strictly increasing sequence $\{x_n\}_{n=1}^\infty$ of real numbers such that the following conditions are satisfied:
\begin{enumerate}
\item the sequence $\{|f(x_n)|\}_{n=1}^\infty$ is strictly increasing,
\item $\lim_{n\to\infty}x_n=\infty$,
\item $\lim_{n\to\infty}|f(x_n)|=\infty$,
\item $|f(x_1)|>0$,
\item $x_1>0$.
\end{enumerate}

We define a function $\phi\colon\mathbb{R}_{>0}\to\mathbb{R}$ by
\[
\phi(x)=\left\{\begin{array}{ll}
\displaystyle\min\left\{\log(x+3),\frac{|f(x_{n+1})|-|f(x_n)|}{x_{n+1}-x_n}(x-x_n)+|f(x_n)|\right\}&\text{if }x_n<x\leq x_{n+1},
\\
\displaystyle\min\left\{\log(x+3),\frac{|f(x_1)|}{2x_1}(x_1+x)\right\}&\text{if }0<x\leq x_1.
\end{array}\right.
\]
Then the function $B\mapsto\sqrt{\phi(B)}$ satisfies Assumption~\ref{3ass}.
It follows that $f(B)=O(\sqrt{\phi(B)})$.
Hence, there exist positive real numbers $C_1$ and $C_2$ such that $|f(B)|\leq C_2\sqrt{\phi(B)}$ for all $B\geq C_1$.
Let $n$ be a positive integer such that $x_n\geq C_1$ and $|f(x_n)|>C_2^2$.
Then we have
\[
|f(x_n)|\leq C_2\sqrt{\phi(x_n)}<\sqrt{|f(x_n)|}\sqrt{\phi(x_n)}\leq|f(x_n)|.
\]
This is a contradiction.
\end{proof}

By applying Lemma~\ref{f} to the function
\[
f(B):=\frac{N(B)}{B^{4k-1}}-\frac{2k\mathcal{G}_S(1,2k-1)}{(3k-1)(4^k-1)|B_{2k}|\zeta(4k-1)}\log B,
\]
we obtain the main theorem.

\begin{theorem}\label{NN}
The asymptotic formula
\[
N(B)=\frac{2k\mathcal{G}_S(1,2k-1)}{(3k-1)(4^k-1)|B_{2k}|\zeta(4k-1)}B^{4k-1}\log B+O(B^{4k-1})
\]
holds.
\end{theorem}
Finally, we compute the constant $\mathcal{G}_S(1,2k-1)=\prod_p\mathcal{G}_p(1,2k-1)$.
By the definition of $\mathcal{G}_p(s,w)$ in Definition~\ref{G}, we obtain the following proposition.

\begin{prop}\label{Gp}
The local factor $\mathcal{G}_p(1,2k-1)$ is given as follows.
\begin{enumerate}
\item When $p\in S\setminus\{2\}$, we have
\[
\mathcal{G}_p(1,2k-1)=\left(1+\frac{2}{p}+\frac{3}{p^{2k}}+\frac{2p+1}{p^{4k}}\right)\left(1-\frac{1}{p}\right)\left(1-\frac{1}{p^{6k-2}}\right)^{-1}.
\]
\item When $p\notin S\cup\{2\}$, we have
\begin{align*}
&\mathcal{G}_p(1,2k-1)
\\
&=\left(1-\frac{1}{p^3}+\frac{2p^2-p-1}{p^{2k+2}}+\frac{p^2-2p+1}{p^{4k+1}}-\frac{2p^2-p-1}{p^{6k+1}}-\frac{p^2+p-2}{p^{8k}}\right)\left(1-\frac{1}{p^{4k-1}}\right)^{-1}\left(1-\frac{1}{p^{6k-2}}\right)^{-1}.
\end{align*}
\item When $2\in S$, we have
\[
\mathcal{G}_2(1,2k-1)=1+\frac{3(-1)^k}{4(2^{2k-1}-1)}-\frac{(-1)^k(2^{2k}-1)(2^{4k-2}+2^{2k-1}+1)}{4(2^{2k-1}-1)(2^{6k-2}-1)}.
\]
\item When $2\notin S$, we have
\[
\mathcal{G}_2(1,2k-1)=\frac{7}{8}+\frac{5(-1)^k}{8(2^{2k-1}-1)}-\frac{(-1)^k(2^{2k}-1)(2^{8k-4}+2^{6k-2}+2^{4k-2}-2^{2k-2}-1)}{4(2^{2k-1}-1)(2^{4k-1}-1)(2^{6k-2}-1)}.
\]
\end{enumerate}
In particular, we have $\mathcal{G}_S(1,2k-1)>0$.
\end{prop}

\section*{Funding}
This work was financially supported by the Japan Science and Technology Agency (JST) through JST SPRING, Grant Number JPMJSP2125, under the THERS Make New Standards Program for the Next Generation Researchers.

\section*{Author contributions}
Haruki Ito conducted the research and wrote the manuscript.

\end{document}